\documentclass[11pt]{amsart}
\usepackage{amssymb,mathrsfs,graphicx,extpfeil}
\usepackage{epsfig}
\usepackage{indentfirst, latexsym, amssymb, enumerate,amsmath,graphicx}
\usepackage{float}
\usepackage{colortbl}
\usepackage{epsfig, subfig, caption}

\usepackage{amsmath}
\title[ Emergent of bi-cluster flocking for the Cucker-Smale ensemble]{Emergent behaviors of the Cucker-Smale ensemble under attractive-repulsive couplings and Rayleigh frictions}

\author[Fang]{Di Fang}
\address[Di Fang]{\newline Department of Mathematics, \newline University of Wisconsin
Madison, WI 53706, USA}
\email{di@math.wisc.edu}

\author[Ha]{Seung-Yeal Ha}
\address[Seung-Yeal Ha]{\newline Department of Mathematical Sciences and Research Institute of Mathematics \newline Seoul National University, Seoul 08826, Republic of Korea \newline
Korea Institute for Advanced Study, Hoegiro 85, Seoul, 02455,  Republic of Korea}
\email{syha@snu.ac.kr}

\author[Jin]{Shi Jin}
\address[Shi Jin]{\newline Department of Mathematics, \newline University of Wisconsin
Madison, WI 53706, USA}
\email{jin@math.wisc.edu}

\newtheorem{theorem}{Theorem}[section]
\newtheorem{lemma}{Lemma}[section]
\newtheorem{corollary}{Corollary}[section]

\newtheorem{example}{Example}[section]
\newtheorem{remark}{Remark}[section]

\newtheorem{definition}{Definition}[section]

\newcommand{\bbr}{\mathbb R}

\def\charf {\mbox{{\text 1}\kern-.30em {\text l}}}

\begin{document}

\date{\today}


\thanks{\textbf{Acknowledgment.} The work of S.-Y. Ha  is supported by the National Research Foundation of
Korea (NRF-2017R1A2B2001864). The work of Di Fang and Shi Jin are supported by NSF grants DMS-1522184 and DMS-1107291: RNMS ``KI-Net".}

\begin{abstract}
In this paper, we revisit an interaction problem of two homogeneous Cucker-Smale (in short C-S) ensembles with attractive-repulsive couplings, possibly under the effect of Rayleigh friction, and study three sufficient frameworks leading to bi-cluster flocking in which two sub-ensembles evolve to two-clusters departing from each other. In previous literature, the interaction problem has been studied in the context of attractive couplings. In our interaction problem, inter-ensemble and intra-ensemble couplings are assumed to be repulsive and attractive respectively. When the Rayleigh frictional forces are turned on, we show that the total kinetic energy is uniformly bounded so that spatially mixed initial configurations evolve toward the bi-cluster configuration asymptotically fast under some suitable conditions on system parameters, communication weight functions and initial configurations. In contrast, when Rayleigh frictional forces are turned off, the flocking analysis is more delicate mainly due to the possibility of an exponential growth of the kinetic energy. In this case, we employ two mutually disjoint frameworks with constant inter-ensemble communication function and exponentially localized inter-ensemble communication functions respectively, and prove the bi-clustering phenomenon in both cases. This work extends the previous work on the interaction problem of C-S ensembles. We also conduct several numerical experiments and compare them with our theoretical results.
\end{abstract}
\maketitle \centerline{\date}


\section{Introduction} \label{sec:1}
\setcounter{equation}{0}
Emergent dynamics of many-body systems are uniquitous in our nature and it has been a renewed interest in recent years due to their possible engineering application in sensor network, control of drones and driverless cars, etc \cite{B-H, D-M, J-K, L-P-L-S, P-E-G, P-L-S-G-P, T-T, V-C-B-C-S}. In this work, we consider the Cucker-Smale flocking model \cite{C-S2} which has been extensively studied in literature \cite{C-H-L}. More precisely, we consider a spatially mixed ensemble consisting of two homogeneous Cucker-Smale (C-S) sub-ensembles denoted by ${\mathcal G}_1$ and ${\mathcal G}_2$ respectively, and in each sub-ensemble ${\mathcal G}_i$, C-S particles interact via the attractive flocking force with communication weight function $\psi_s$, whereas C-S particles between different ensembles interact by repulsive flocking force with communication weight function $\psi_d$. In this situation, we are interested in the emergence of bi-cluster flocking from an initially mixed ensemble. More precisely, let $(\textbf{x}_i, \textbf{v}_i)$ and $(\textbf{y}_j, \textbf{w}_j)$ be the position-velocity configuration of the $i$-th and $j$-th particles in ${\mathcal G}_1$ and ${\mathcal G}_2$, respectively. We set the number of particles in each group $|{\mathcal G}_1| =N_1,~|{\mathcal G}_2|  = N_2$. Then, the dynamics of the mixed ensemble ${\mathcal G}_1 \cup {\mathcal G}_2$ are given by the ordinary differential equations:
\begin{align}
\begin{aligned} \label{A-1}
\dot{\textbf{x}}_i &= \textbf{v}_i, \quad t > 0, \quad  i = 1, \cdots, N_1, \\
\dot{\textbf{v}}_i &= \frac{\kappa_{s}}{N_1} \sum_{k=1}^{N_1} \psi_{s}(\|\textbf{x}_k - \textbf{x}_i\|) (\textbf{v}_k -\textbf{v}_i) -\frac{\kappa_{d}}{N_2} \sum_{k=1}^{N_2} \psi_{d}(\|\textbf{y}_k - \textbf{x}_i\|) (\textbf{w}_k - \textbf{v}_i) + \delta \textbf{v}_i (1 - \|\textbf{v}_i\|^2), \\
\dot{\textbf{y}}_j &= \textbf{w}_j, \quad t > 0, \quad j = 1, \cdots, N_2, \\
\dot{\textbf{w}}_j &= \frac{\kappa_{s}}{N_2} \sum_{k=1}^{N_2} \psi_{s}(\|\textbf{y}_k - \textbf{y}_j\|) (\textbf{w}_k - \textbf{w}_j) -\frac{\kappa_{d}}{N_1} \sum_{k=1}^{N_1} \psi_{d}(\| \textbf{x}_k - \textbf{y}_j \|) (\textbf{v}_k - \textbf{w}_j) + \delta \textbf{w}_j (1 - \|\textbf{w}_j\|^2),
\end{aligned}
\end{align}
where $\kappa_s$ and $\kappa_d$ are nonnegative inter and intra coupling strengths respectively, and $\delta$ is a nonnegative constant proportional to the Rayleigh friction. The Lipschitz continuous functions $\psi_s$ and $\psi_d$ represent communication weight functions: there exist positive constants $\psi_s^{\infty}$ and $\psi_d^{\infty}$ such that for $\ell = s, d$,
\begin{align*}
\begin{aligned} \label{A-2}
& 0 \leq \psi_\ell (r) \leq \psi_\ell^{\infty},~r \geq 0 \quad \mbox{and} \quad \lim_{r \rightarrow \infty} \psi_d (r) = 0,\\
&\quad (\psi_{\ell}(r_2) - \psi_{\ell}(r_1)) (r_2 - r_1) \leq 0, \quad r_1, r_2 \geq 0.
\end{aligned}
\end{align*}
Note that the first two terms on the right hand sides of \eqref{A-1} represent  attractive (repulsive) interactions between C-S particles in the same (different) groups, and the last term is the Rayleigh friction force. When the repulsive force is turned off (i.e., $\kappa_d = 0$), system \eqref{A-1} with $\delta > 0$ is a juxtaposition of two C-S systems with Rayleigh friction and they have been studied in literature \cite{H-H-K}. In this paper, we consider the situation where both attractive and repulsive interactions exist simultaneously, i.e.,
\begin{equation*} \label{A-3}
\kappa_s > 0 \quad \mbox{and} \quad \kappa_d > 0,
\end{equation*}
and we would like to justify the emergent dynamics of bi-cluster flocking from spatially mixed initial configurations. For the simplicity of presentation, we set
\[
X := (\textbf{x}_1, \cdots, \textbf{x}_N), \quad Y := (\textbf{y}_1, \cdots, \textbf{y}_N), \quad V := (\textbf{v}_1, \cdots, \textbf{v}_N), \quad W := (\textbf{w}_1, \cdots, \textbf{w}_N),
\]
and we also recall the definition of bi-cluster flocking as follows.
\begin{definition} \label{D1.1}
\emph{\cite{C-H-H-J-K}}
Let $\{ (X, V),~(Y, W) \}$ be a solution to system \eqref{A-1}. Then, the solution tends to the bi-cluster flocking asymptotically if the following relations holds.
\begin{enumerate}
\item
Each sub-ensemble exhibits mono-cluster flocking asymptotically:
\begin{align*}
\begin{aligned}
& \sup_{0 \leq t < \infty}  \max_{1 \leq i, j \leq N_1} \|\textbf{x}_i(t) - \textbf{x}_j(t) \| < \infty, \quad \lim_{t \to \infty}  \max_{1 \leq i, j \leq N_1} \|\textbf{v}_i(t) - \textbf{v}_j(t) \| = 0, \\
& \sup_{0 \leq t < \infty}  \max_{1 \leq i, j \leq N_2} \|\textbf{y}_i(t) - \textbf{y}_j(t) \| < \infty, \quad \lim_{t \to \infty}  \max_{1 \leq i, j \leq N_2} \|\textbf{w}_i(t) - \textbf{w}_j(t) \| = 0.
\end{aligned}
\end{align*}
\item
The two sub-ensembles  separate from each other asymptotically: there exists a pair $i, j$ such that
\[ \lim_{t \to \infty} \|\textbf{x}_i(t) - \textbf{y}_j(t) \| = \infty.   \]
\end{enumerate}
\end{definition}
\begin{remark}
1. Note that the relation
\[ \inf_{0 \leq t < \infty} \min_{i, j} \| \textbf{v}_i(t) - \textbf{w}_j(t) \| > 0 \]
implies asymptotic separation of two clusters. \newline

\noindent 2. For the mono-cluster flocking to the C-S model, there are lots of literature \cite{A-H, C-F-R-T, C-H-L,C-M, C-S2, D-F-T, H-L-L, H-Liu, H-T, L-X, M-T1, M-T2} under diverse physical situations.
\end{remark}
The main question to be explored in this paper is to study dynamic patterns arising from the interaction of two homogeneous C-S ensembles. In fact, this interaction problem between two homogeneous C-S ensembles has already been addressed in \cite{H-K-Z-Z1, H-K-Z-Z2, H-K-Z} from {\it spatially well-separated initial configuration and attractive couplings only}. Hence, compared to the aforementioned works, this paper has two novelties. First, we relaxed our admissible initial configurations to be mixed so that how many clusters will emerge from the initial configuration is not clear a priori. Second, we allow our inter-ensemble interactions to be repulsive, whereas the coupling inside the same homogeneous ensemble is attractive. The presence of attractive and repulsive couplings at the same time make analysis much harder than the pure attractive case. \newline

The main goal of the paper is to provide three distinct frameworks leading to the bi-cluster flocking for system \eqref{A-1} with spatially mixed initial configuration. Below, we set
\[ D(X) := \max_{1 \leq i, j \leq N_1} \| \textbf{x}_i - \textbf{x}_j\|, \qquad D(V) :=  \max_{1 \leq i, j \leq N_1} \| \textbf{v}_i - \textbf{v}_j\|.  \]
In our first framework, we take system parameters, communication weight functions and initial configurations to satisfy
\[
\begin{cases}
\displaystyle \kappa_s > 0, \quad \kappa_d > 0, \quad D(X(0)) > 0, \quad D(Y(0)) > 0, \quad \| \textbf{v}_{c}(0) - \textbf{w}_{c}(0) \| > 0, \quad \delta = 0, \\
\displaystyle \psi_d \equiv 1, \quad D(V(0)) < \kappa_s \int_{D(X(0))}^{\infty} \psi_s(r) dr, \quad  D(W(0)) < \kappa_s \int_{D(Y(0))}^{\infty} \psi_s(r) dr,
\end{cases}
\]
 In this setting, the inter-communication function is a {\it constant}, and the local averages and local fluctuations are completely decoupled so that the dynamics of local averages is solvable (see Lemma \ref{L3.1}), whereas the local fluctuations are also completely decoupled so that each sub-ensemble behaves like the C-S model itself without knowing the other sub-ensemble. Thus, we can apply the Lyapunov functional approach developed for the mono-cluster flocking of the C-S model to each sub-ensemble to get sufficient conditions (see Theorem \ref{T3.1}).  In particular our sufficient framework yields that if $\psi_s$ is long-ranged, i.e., $\int_0^{\infty} \psi_s(r) dr = \infty$, then for any initial configuration, we have a bi-cluster flocking.
 
 In the second framework, we consider an exponentially {\it localized} inter communication weight functions
\[ \psi_s(r) \approx 1, \quad \psi_d(r) \ll e^{-\beta \kappa_d \psi_d^{\infty} r}, \quad \kappa_s \gg \kappa_d > 0. \]
For more precise description for the framework, we refer to Theorem \ref{T4.1}. In this setting, again we show that from spatially mixed initial configuration, bi-cluster flocking will emerge asymptotically.  

In the third framework, we consider system \eqref{A-1} with the following setup:
\[ \delta > 0, \quad \kappa_s \gg \kappa_d \]
In this case, we can show that the total kinetic energy which is the second velocity moment is uniformly bounded, thanks to the effect of the Rayleigh friction. With this uniform bound for the second velocity moment, we can show that spatially mixed initial configuration evolves toward a bi-cluster flocking state. In this relaxation process, the local velocity fluctuations decay to zero exponentially fast, whereas the local average positions of each sub-ensemble move away at least linearly in time $t$ (see Theorem \ref{T5.1} for details). \newline

The rest of this paper is organized as follows. In Section \ref{sec:2}, we study a priori estimates for \eqref{A-1} such as propagation of the first two velocity moments and reformulation of the macro-micro dynamics of \eqref{A-1}. In Section \ref{sec:3}, we present our first framework for bi-cluster flocking where constant inter communication function and zero Rayleigh friction are employed. In Section \ref{sec:4}, we consider exponential localized inter communication weight function. In this case, we provide a priori sufficient framework leading to the bi-cluster flocking. In Section \ref{sec:5}, we consider system \eqref{A-1} with a positive Rayleigh friction. In this case, under less restrictive conditions compared to the second framework in Section \ref{sec:4}, we show that bi-cluster flocking will emerge from spatially mixed initial configuration. In Section \ref{sec:6}, we conduct several numerical experiments to illustrate our theoretical results in previous sections and compare them with numerical results. Finally, Section \ref{sec:7} is devoted to a brief summary of our main results and remaining issues to be explored in future works.  In Appendix, we briefly present several Gronwall type lemmas which serves as needed ingredients in the flocking analysis.

\bigskip

\noindent {\bf Notation}: We use simplified notation for a double sum:
\[ \sum_{i,k = 1}^{N_\ell} := \sum_{i =1}^{N_\ell} \sum_{k =1}^{N_\ell}, \quad \ell = 1, 2. \]

\section{Preliminaries} \label{sec:2}
\setcounter{equation}{0}
In this section, we study two preparatory materials ``{\it propagation of velocity moments}" and ``{\it macro-micro decomposition}" of system \eqref{A-1} which will be used crucially in later sections.

\subsection{Propagation of velocity moments} \label{sec:2.1}
We first introduce normalized velocity moments for a velocity configuration $(V, W)$:
\begin{align*}
\begin{aligned}
M_1(V) &:= \frac{1}{N_1} \sum_{i=1}^{N_1} \textbf{v}_i, \quad  M_1(W) := \frac{1}{N_2} \sum_{i=1}^{N_2} \textbf{w}_i, \quad M_1 := M_1(V) + M_1(W),  \\
M_2(V) &:= \frac{1}{N_1} \sum_{i=1}^{N_1} \|\textbf{v}_i\|^2, \quad  M_2(W) := \frac{1}{N_2}  \sum_{i=1}^{N_2} \|\textbf{w}_i\|^2, \quad M_2 := M_2(V)  +  M_2(W).
\end{aligned}
\end{align*}
\begin{lemma} \label{L2.1}
Let $\{ (X, V), (Y, W) \}$ be a solution to \eqref{A-1}. Then $M_1$ and $M_2$ satisfy
\begin{eqnarray*}
&& (i)~ \frac{dM_1}{dt}  = \delta \Big[  \frac{1}{N_1} \sum_{i=1}^{N_1} \textbf{v}_i (1-\|\textbf{v}_i \|^2) + \frac{1}{N_2} \sum_{i=1}^{N_2} \textbf{w}_i (1-\|\textbf{w}_i\|^2) \Big]. \cr
&& (ii)~\frac{dM_2}{dt} =  - \kappa_s \Big[ \frac{1}{N^2_1}  \sum_{i, k=1}^{N_1} \psi_{s}(\|\textbf{x}_k - \textbf{x}_i \|) \|\textbf{v}_k - \textbf{v}_i\|^2 + \frac{1}{N^2_2}  \sum_{j, k=1}^{N_2} \psi_{s}(\|\textbf{y}_k - \textbf{y}_j\|) \|\textbf{w}_k - \textbf{w}_j\|^2 \Big] \\
&& \hspace{2cm} + \frac{2\kappa_{d}}{N_1 N_2} \sum_{i=1}^{N_1} \sum_{j=1}^{N_2}
\psi_{d}(\|\textbf{y}_j - \textbf{x}_i\|) \|\textbf{v}_i - \textbf{w}_j\|^2  + \frac{2 \delta}{N_1} \sum_{i=1}^{N_1}  \|\textbf{v}_i \|^2 (1 - \|\textbf{v}_i\|^2)  \\
&& \hspace{2cm} +  \frac{2 \delta}{N_2} \sum_{j=1}^{N_2}  \| \textbf{w}_j\|^2 (1 - \|\textbf{w}_j\|^2).
\end{eqnarray*}
\end{lemma}
\begin{proof}
(i)~We add $\eqref{A-1}_2$ over all $i$ and divide the resulting relation by $N_1$ to get
\begin{align}
\begin{aligned} \label{B-1}
\frac{d}{dt} M_1(V) &= \frac{\kappa_{s}}{N^2_1} \sum_{k, i=1}^{N_1} \psi_{s}(\|\textbf{x}_k - \textbf{x}_i\|) (\textbf{v}_k - \textbf{v}_i) \\
&\hspace{0.2cm} -\frac{\kappa_{d}}{N_1 N_2} \sum_{i=1}^{N_1} \sum_{k=1}^{N_2} \psi_{d}(\|\textbf{y}_k - \textbf{x}_i\|) (\textbf{w}_k - \textbf{v}_i) + \frac{\delta}{N_1} \sum_{i=1}^{N_1} \textbf{v}_i (1 - \|\textbf{v}_i\|^2)  \\
&= -\frac{\kappa_{d}}{N_1 N_2} \sum_{i=1}^{N_1} \sum_{k=1}^{N_2} \psi_{d}(\|\textbf{y}_k - \textbf{x}_i|) (\textbf{w}_k - \textbf{v}_i) + \frac{\delta}{N_1} \sum_{i=1}^{N_1} \textbf{v}_i (1 - \|\textbf{v}_i\|^2),
\end{aligned}
\end{align}
where we used a symmetry trick $i \leftrightarrow j$ to see that the first term in the right hand side of \eqref{B-1} is zero. By the same argument, 
\begin{equation} \label{B-2}
\frac{d}{dt} M_1(W) = -\frac{\kappa_{d}}{N_1 N_2} \sum_{j=1}^{N_2} \sum_{k=1}^{N_1} \psi_{d}(\|\textbf{x}_k - \textbf{y}_j\|) (\textbf{v}_k - \textbf{w}_j)
+ \frac{\delta}{N_2} \sum_{j=1}^{N_2} \textbf{w}_j (1 - \|\textbf{w}_j\|^2).
\end{equation}
Now, adding \eqref{B-1} and \eqref{B-2} and using the relabeling trick give
\begin{align*}
\begin{aligned} 
& \frac{dM_1}{dt} =  -\frac{\kappa_{d}}{N_1 N_2} \Big[  \sum_{i=1}^{N_1} \sum_{k=1}^{N_2} \psi_{d}(\|\textbf{y}_k - \textbf{x}_i\|) (\textbf{w}_k - \textbf{v}_i)  + \sum_{j=1}^{N_2} \sum_{k=1}^{N_1} \psi_{d}(\|\textbf{x}_k - \textbf{y}_j\|) (\textbf{v}_k - \textbf{w}_j) \Big]\\
& \hspace{1.5cm}  + \frac{\delta}{N_1} \sum_{i=1}^{N_1} \textbf{v}_i (1 - \|\textbf{v}_i\|^2) + \frac{\delta}{N_2} \sum_{j=1}^{N_2} \textbf{w}_j (1 - \|\textbf{w}_j\|^2) \\
&  \hspace{1cm}  =\frac{\delta}{N_1} \sum_{i=1}^{N_1} \textbf{v}_i (1 - \|\textbf{v}_i\|^2) + \frac{\delta}{N_2} \sum_{j=1}^{N_2} \textbf{w}_j (1 - \|\textbf{w}_j\|^2).
\end{aligned}
\end{align*}

\vspace{0.5cm}

\noindent (ii)~For the estimate of $M_2(V)$, we take an inner product $\eqref{A-1}_2$ with $2\textbf{v}_i$, sum it over all $i$ and divide the resulting relation by $N_1$ to obtain
\begin{align}
\begin{aligned} \label{B-4}
&\frac{d}{dt} M_2(V) = \frac{2\kappa_{s}}{N^2_1}  \sum_{i, k=1}^{N_1} \psi_{s}(\|\textbf{x}_k - \textbf{x}_i\|) \textbf{v}_i \cdot (\textbf{v}_k - \textbf{v}_i)  \\
& \hspace{1cm} -\frac{2\kappa_{d}}{N_1 N_2} \sum_{i=1}^{N_1} \sum_{k=1}^{N_2}  \psi_{d}(|\textbf{y}_k - \textbf{x}_i \|) \textbf{v}_i \cdot (\textbf{w}_k - \textbf{v}_i) +  \frac{2 \delta}{N_1} \sum_{i=1}^{N_1}  \|\textbf{v}_i\|^2 (1 - \|\textbf{v}_i\|^2) \\
&  \hspace{1cm} =  -\frac{\kappa_{s}}{N^2_1}  \sum_{i, k=1}^{N_1} \psi_{s}(\| \textbf{x}_k - \textbf{x}_i\|) \|\textbf{v}_k - \textbf{v}_i\|^2 -\frac{2\kappa_{d}}{N_1 N_2} \sum_{i=1}^{N_1} \sum_{k=1}^{N_2} \psi_{d}(\|\textbf{y}_k - \textbf{x}_i \|) \textbf{v}_i \cdot (\textbf{w}_k - \textbf{v}_i) \\
& \hspace{1cm} +  \frac{2 \delta}{N_1} \sum_{i=1}^{N_1}  \|\textbf{v}_i\|^2 (1 - \|\textbf{v}_i\|^2).
\end{aligned}
\end{align}

Similarly, 
\begin{align}
\begin{aligned} \label{B-5}
&\frac{d}{dt} M_2(W) =  -\frac{\kappa_{s}}{N^2_2}  \sum_{j, k=1}^{N_2} \psi_{s}(\|\textbf{y}_k - \textbf{y}_j\|) |\textbf{w}_k - \textbf{w}_j|^2 \\
& \hspace{0.5cm} -\frac{2\kappa_{d}}{N_1 N_2} \sum_{j=1}^{N_2} \sum_{k=1}^{N_1} \psi_{d}(\|\textbf{x}_k - \textbf{y}_j \|) \textbf{w}_j \cdot (\textbf{v}_k - \textbf{w}_j) +  \frac{2 \delta}{N_2} \sum_{j=1}^{N_2} \|\textbf{w}_j \|^2 (1 - \|\textbf{w}_j\|^2).
\end{aligned}
\end{align}
Finally, one combines \eqref{B-4} and \eqref{B-5} and uses summation index exchanges to get the desired estimate.
\end{proof}
\begin{remark} \label{R2.1}
Note that for $\kappa_d > 0$ and $\delta >0$, the total sum of normalized energies may not be monotonically decreasing. In fact, it may grow exponentially for the special case with $\kappa_s = 0$ and $\delta = 0$ (see Section \ref{sec:3.1} and Section \ref{sec:3.2}). This is why we do not expect the emergence of mono-cluster flocking in general.
\end{remark}
As will be seen in Section \ref{sec:3.1}, $M_2$ may grow exponentially in general, when the Rayleigh friction term is turned off.  However, for the case with $\delta > 0$, $M_2$ is uniformly bounded as we will see below. This is one of virtue of the nonlinear frictional force.
\begin{corollary} \label{C2.1}
Suppose that the system parameters satisfy
\[ \kappa_s \geq 0, \quad  \kappa_d > 0, \quad \delta > 0, \]
and let $\{ (X, V), (Y, W) \}$ be a solution to \eqref{A-1}. Then, $M_2$ is uniformly bounded: there exists a positive constant $M_2^{\infty} = M_2^{\infty}(\kappa_d, \psi_d^{\infty}, \delta, M_2(0))$ such that
\[  \sup_{0 \leq t < \infty} M_2(t) \leq M_2^{\infty} < \infty.     \]
\end{corollary}
\begin{proof} It follows from Lemma \ref{L2.1} that
\begin{align}
\begin{aligned} \label{B-5-1}
\frac{dM_2}{dt} &\leq \frac{2\kappa_{d} \psi_d^{\infty}}{N_1 N_2} \sum_{i=1}^{N_1} \sum_{j=1}^{N_2}\|\textbf{v}_i - \textbf{w}_j\|^2 + \frac{2 \delta}{N_1} \sum_{i=1}^{N_1}  \|\textbf{v}_i\|^2 (1 - \|\textbf{v}_i\|^2) \\
& \hspace{0.2cm} +  \frac{2 \delta}{N_2} \sum_{j=1}^{N_2} \|\textbf{w}_j \|^2 (1 - \|\textbf{w}_j\|^2).
\end{aligned}
\end{align}
On the other hand, it follows from the Cauchy-Schwarz inequality that
\begin{align*}
\begin{aligned}
& \Big( \sum_{i=1}^{N_1} \|\textbf{v}_i\|^2 \Big)^2 \leq N_1 \sum_{i=1}^{N_1} \|\textbf{v}_i\|^4, \quad
\Big( \sum_{j=1}^{N_2} \|\textbf{w}_j\|^2 \Big)^2 \leq N_2 \sum_{j=1}^{N_2} \|\textbf{w}_j\|^4, \\
& \mbox{and} \quad  \| \textbf{v}_i - \textbf{w}_j\|^2 \leq 2 ( \|\textbf{v}_i \|^2 + \|\textbf{w}_j\|^2 ).
\end{aligned}
\end{align*}
These relations and \eqref{B-5-1} yield a Riccati type differential inequality:
\begin{equation} \label{B-6}
\frac{dM_2}{dt} \leq (4 \kappa_d \psi_d^{\infty} + 2\delta) M_2 -\delta (M_2)^2.
\end{equation}
Let $y$ be a solution of the following Riccati equation:
\begin{equation} \label{B-7}
 y^{\prime} = (4 \kappa_d \psi_d^{\infty} + 2\delta) y -\delta y^2, \quad t > 0, \qquad y(0) = M_2(0).
\end{equation}
Then, we use phase line analysis to see
\[ \sup_{0 \leq t < \infty} y(t) \leq \max \Big \{  2 + \frac{4 \kappa_d \psi_d^{\infty}}{\delta},~M_2(0) \Big \}.     \]
By the comparison principle of ODE between \eqref{B-6} and \eqref{B-7}, one has
\[ M_2(t) \leq y(t) \leq \max \Big \{  2 + \frac{4 \kappa_d \psi_d^{\infty}}{\delta},~M_2(0) \Big \} =: M_2^{\infty}  \]
which yields the desired estimate.
\end{proof}
\begin{remark} \label{R2.2}
Note that $\delta > 0$ is crucially used to get the uniform bound of $M_2$. If $\delta = 0$, it follows from \eqref{B-6} that 
\[ \frac{dM_2}{dt} \leq (4 \kappa_d \psi_d^{\infty}) M_2, \quad \mbox{i.e.,} \quad  M_2(t) \leq M_2(0) e^{4 \kappa_d \psi_d^{\infty} t}, \quad t \geq 0.\]
Thus, the upper bound of $M_2$ can grow exponentially. This can be seen explicitly in the two-particle and three-particle systems in Section \ref{sec:3}.
\end{remark}

\subsection{The micro-macro decomposition} \label{sec:2.2} For given ensemble $\{ (X, V), (Y, W) \}$, we set local averages and fluctuations around them:
\begin{align*}
\begin{aligned} \label{C-1}
& \textbf{x}_{c} := \frac1{N_1}\sum_{i=1}^{N_1}\textbf{x}_{i}, \quad  \textbf{y}_{c} := \frac1{N_2}\sum_{j=1}^{N_2}\textbf{y}_{j}, \quad \textbf{v}_{c} := \frac1{N_1}\sum_{i=1}^{N_1}\textbf{v}_{i}, \quad  \textbf{w}_{c} := \frac1{N_2}\sum_{j=1}^{N_2}\textbf{w}_{j}, \\
& \hat{\textbf{x}}_{i} := \textbf{x}_{i}-\textbf{x}_{c}, \qquad  \hat{\textbf{y}}_{i} := \textbf{y}_{i}-\textbf{y}_{c}, \qquad  \hat{\textbf{v}}_{i} := \textbf{v}_{i}-\textbf{v}_{c}, \qquad  \hat{\textbf{w}}_{i} := \textbf{w}_{i}-\textbf{w}_{c}.
\end{aligned}
\end{align*}
In analogy with kinetic theory, we call $(\textbf{x}_{c},  \textbf{v}_{c})$ and $(  \hat{\textbf{x}}_{i} ,  \hat{\textbf{v}}_{i} )$ as ``{\it macro}" and ``{\it micro}" components of the state $(\textbf{x}_i, \textbf{v}_i)$, respectively. Next, we study the dynamics of macro and micro components.
\begin{lemma} \label{L2.2}
Let $\{ (X, V), (Y, W) \}$ be a solution to \eqref{A-1}. Then, the micro-macro dynamics are given by the coupled system:
\begin{equation}
\begin{cases} \label{C-2}
\displaystyle \dot{\textbf{x}}_{c} = \textbf{v}_{c}, \quad \dot{\textbf{y}}_{c} = \textbf{w}_{c}, \quad t > 0,\\
\displaystyle \dot{\textbf{v}}_{c}
= -\frac {\kappa_d}{N_1N_2} \sum_{i=1}^{N_1} \sum_{k=1}^{N_2} \psi_d(\|\textbf{y}_{k}
- \textbf{x}_{i} \|)\big(\textbf{w}_{k} - \textbf{v}_{i}\big) +  \frac{\delta}{N_1} \sum_{i=1}^{N_1} \textbf{v}_i ( 1 - \|\textbf{v}_i\|^2), \\
\displaystyle \dot{\textbf{w}}_{c}
= -\frac {\kappa_d}{N_1N_2} \sum_{j=1}^{N_2} \sum_{k=1}^{N_1} \psi_d(\|\textbf{x}_{k}
- \textbf{y}_{j} \|)\big(\textbf{v}_{k} - \textbf{w}_{j}\big) +  \frac{\delta}{N_2} \sum_{i=1}^{N_2} \textbf{w}_i ( 1 - \|\textbf{w}_i\|^2),
\end{cases}
\end{equation}
and
\begin{equation}\label{C-3}
\begin{cases}
\displaystyle \dot{\hat{\textbf{x}}}_{i} = \hat{\textbf{v}}_{i}, \quad \dot{\hat{\textbf{y}}}_{i} = \hat{\textbf{w}}_{i}, \quad  t\geq0, \\
\displaystyle \dot{\hat{\textbf{v}}}_{i} =  - \dot{\textbf{v}}_{c} + \frac{\kappa_s}{N_1 }\sum_{k=1}^{N_1}\psi_s(\|\hat{\textbf{x}}_{k} - \hat{\textbf{x}}_{i}\|)\big(\hat{\textbf{v}}_{k} - \hat{\textbf{v}}_{i}\big)
- \frac{\kappa_d}{N_2} \sum_{k=1}^{N_2}\psi_d(\|\textbf{y}_{k} - \textbf{x}_{i}\|)\big(\textbf{w}_{k} -  \textbf{v}_{i} \big) \\
\displaystyle \hspace{0.6cm} +~\delta \textbf{v}_{i} (1 - \|\textbf{v}_{i}\|^2), \\
\displaystyle \dot{\hat{\textbf{w}}}_{j} =  - \dot{\textbf{w}}_{c} + \frac{\kappa_s}{N_2 }\sum_{k=1}^{N_2}\psi_s(\|\hat{\textbf{y}}_{k} - \hat{\textbf{y}}_{j}\|)\big(\hat{\textbf{w}}_{k} - \hat{\textbf{w}}_{j}\big) - \frac{\kappa_d}{N_1} \sum_{k=1}^{N_1}\psi_d(\|\textbf{x}_{k} - \textbf{y}_{j}\|)\big(\textbf{v}_{k} -  \textbf{w}_{j} \big) \\
\displaystyle \hspace{0.6cm} +~\delta \textbf{w}_{j} (1 - \|\textbf{w}_{j}\|^2).
\end{cases}
\end{equation}
\end{lemma}
\begin{proof} $\bullet$ (The macroscopic dynamics): The derivation of the first two equations are almost trivial. So, let us focus on the derivation of $\dot{\textbf{v}}_{c}$. For this, we add all $i =1, \cdots, N_1$ and divide the resulting relation by $N_1$ to get
\begin{align}
\begin{aligned} \label{C-4}
\dot{\textbf{v}}_{c} = & \frac{1}{ N_1} \sum_{i=1}^{N_1} \dot{\textbf{v}}_{i} \\
=& \frac {\kappa_s}{N_1^2}\sum_{k=1}^{N_1}\sum_{i=1}^{N_1} \psi_s(\|\hat{\textbf{x}}_{k} - \hat{\textbf{x}}_{i}\|)\big(\textbf{v}_{k} - \textbf{v}_{i} \big) - \frac {\kappa_d}{N_1 N_2} \sum_{k=1}^{N_{2}}\sum_{i=1}^{N_1}\psi_d(\|\textbf{y}_{k} - \textbf{x}_{i}\|) \big(\textbf{w}_{k} - \textbf{v}_{i}\big) \\
&+ \frac{\delta}{N_1} \sum_{i=1}^{N_1} \textbf{v}_i ( 1 - \|\textbf{v}_i\|^2).
\end{aligned}
\end{align}
Note that due to the skew-symmetric property of $ \psi_s(\|\textbf{x}_{k} - \textbf{x}_{i}\|)\big(\textbf{v}_{k} - \textbf{v}_{i} \big)$ in the exchange of $i \leftrightarrow k$, the first term in the right hand side of \eqref{C-4} becomes zero, which yields $\eqref{C-3}_2$. The derivation of $\eqref{C-3}_3$ can be done similarly.

\vspace{0.2cm}

\noindent $\bullet$ (The microscopic dynamics): We use $\textbf{v}_{i} =  \textbf{v}_{c} +  \hat{\textbf{v}}_{i}$ and $\eqref{A-1}_2$ to find
\begin{align*}
\begin{aligned}
\dot{\hat{\textbf{v}}}_{i} &= \dot{\textbf{v}}_{i} - \dot{\textbf{v}}_{c} \notag\\
&=  - \dot{\textbf{v}}_{c} + \frac{\kappa_s}{N_1} \sum_{k=1}^{N_1} \psi_s(\|\hat{\textbf{x}}_{k} - \hat{\textbf{x}}_{i}\|) \big(\hat{\textbf{v}}_{k} - \hat{\textbf{v}}_{i}\big) -\frac{\kappa_d}{N_2}
\sum_{k=1}^{N_2}\psi_d(\|\textbf{y}_{k} - \textbf{x}_{i}\|) \big(\textbf{w}_{k} - \textbf{v}_{i}\big) \\
& + \delta  \textbf{v}_{i} (1 - \|\textbf{v}_{i}\|^2).
\end{aligned}
\end{align*}
The other case can be calculated similarly.
\end{proof}
Before we close this section, we derive the differential equation for $\textbf{v}_c - \textbf{w}_c$ as follows.
\begin{lemma} \label{L2.3}
Let $\{ (X, V), (Y, W) \}$ be a solution to \eqref{A-1}.  Then $\textbf{v}_c - \textbf{w}_c$ satisfies
\begin{align}
\begin{aligned} \label{C-5}
&\dot{\textbf{v}}_{c} - \dot{\textbf{w}}_{c} = \frac {2\kappa_d}{N_1N_2} \sum_{j=1}^{N_2} \sum_{i=1}^{N_1} \Big( \psi_d(\|\textbf{x}_{i}
- \textbf{y}_{j} \|)  + \delta \Big)  ( \textbf{v}_c -  \textbf{w}_c)  \\
& \hspace{0.2cm} +  \frac {2\kappa_d}{N_1N_2} \sum_{i=1}^{N_1} \sum_{j=1}^{N_2} \psi_d(\|\textbf{y}_{j}
- \textbf{x}_{i} \|)\big(\hat{\textbf{v}}_{i} - \hat{\textbf{w}}_{j} \big) -\frac{\delta}{N_1} \sum_{i=1}^{N_1} \textbf{v}_i  \|\textbf{v}_i\|^2 + \frac{\delta}{N_2} \sum_{i=1}^{N_2} \textbf{w}_i  \|\textbf{w}_i \|^2.
\end{aligned}
\end{align}
\end{lemma}
\begin{proof} It follows from \eqref{C-2} that
\begin{align*}
\begin{aligned}
\dot{\textbf{v}}_{c} - \dot{\textbf{w}}_{c}  = &-\frac {\kappa_d}{N_1N_2} \sum_{i=1}^{N_1} \sum_{k=1}^{N_2} \psi_d(\|\textbf{y}_{k}
- \textbf{x}_{i} \|)\big(\textbf{w}_{k} - \textbf{v}_{i}\big) + \frac {\kappa_d}{N_1N_2} \sum_{j=1}^{N_2} \sum_{k=1}^{N_1} \psi_d(\|\textbf{x}_{k}
- \textbf{y}_{j} \|)\big(\textbf{v}_{k} - \textbf{w}_{j}\big)  \\
&+ \frac{\delta}{N_1} \sum_{i=1}^{N_1} \textbf{v}_i ( 1 - \|\textbf{v}_i\|^2) - \frac{\delta}{N_2} \sum_{i=1}^{N_2} \textbf{w}_i ( 1 - \|\textbf{w}_i\|^2) \\
= &-\frac {\kappa_d}{N_1N_2} \sum_{i=1}^{N_1} \sum_{k=1}^{N_2} \psi_d(\|\textbf{y}_{k}
- \textbf{x}_{i} \|)\big( \textbf{w}_c - \textbf{v}_c + \hat{\textbf{w}}_{k} - \hat{\textbf{v}}_{i}\big) \\
&+ \frac {\kappa_d}{N_1N_2} \sum_{j=1}^{N_2} \sum_{k=1}^{N_1} \psi_d(\|\textbf{x}_{k}
- \textbf{y}_{j} \|)\big(  \textbf{v}_c -  \textbf{w}_c + \hat{\textbf{v}}_{k} - \hat{\textbf{w}}_{j}\big) \\
&+ \delta ( \textbf{v}_c -  \textbf{w}_c) - \frac{\delta}{N_1} \sum_{i=1}^{N_1} \textbf{v}_i  \|\textbf{v}_i\|^2 + \frac{\delta}{N_2} \sum_{i=1}^{N_2} \textbf{w}_i  \|\textbf{w}_i\|^2.
\end{aligned}
\end{align*}
This yields the desired estimate.
\end{proof}

\section{Constant inter-communication and zero Rayleigh friction}  \label{sec:3}
\setcounter{equation}{0}
In this section, we study the emergent dynamics of system \eqref{A-1} with {\it constant} inter-communication:
\begin{equation} \label{CD-0}
\psi_s(r) \geq 0, \quad  (\psi_s(r_2) - \psi_s(r_1)) (r_2 - r_1) \leq 0, \quad r_1, r_2 \geq 0, \quad \psi_d \equiv 1, \quad \delta = 0.
\end{equation}
In this setting,  system \eqref{A-1} can be rewritten as follows.
\begin{align}
\begin{aligned} \label{CD-1}
\dot{\textbf{x}}_i &= \textbf{v}_i, \quad i = 1, \cdots, N_1, \\
\dot{\textbf{v}}_i &= \frac{\kappa_{s}}{N_1} \sum_{k=1}^{N_1} \psi_{s}(|\textbf{x}_k - \textbf{x}_i|) (\textbf{v}_k - \textbf{v}_i) -\frac{\kappa_{d}}{N_2} \sum_{k=1}^{N_2} (\textbf{w}_k - \textbf{v}_i), \\
\dot{\textbf{y}}_j &= \textbf{w}_j, \quad j = 1, \cdots, N_2, \\
\dot{\textbf{w}}_j &= \frac{\kappa_{s}}{N_2} \sum_{k=1}^{N_2} \psi_{s}(|\textbf{y}_k - \textbf{y}_j|) (\textbf{w}_k - \textbf{w}_j) -\frac{\kappa_{d}}{N_1} \sum_{k=1}^{N_1} (\textbf{w}_k - \textbf{v}_i).
\end{aligned}
\end{align}
Before we deal with the above many-body system, we first begin with two or three-body systems as a warm-up problem.

\subsection{A small system with constant communication weights} \label{sec:3.1} First, we consider the two-particle system with $N_1 = N_2 = 1$:
\begin{align}
\begin{aligned} \label{CD-1-1}
& \dot{\textbf{x}} = \textbf{v}, \quad \dot{\textbf{y}}  = \textbf{w},  \quad t > 0, \\
& \dot{\textbf{v}}=  -\kappa_d (\textbf{w} - \textbf{v}), \quad {\dot w} = -\kappa_d (\textbf{v} - \textbf{w}).
\end{aligned}
\end{align}
To reduce the number of equations in \eqref{CD-1-1}, consider spatial and velocity differences:
\[ \textbf{z} := \textbf{x} - \textbf{y}, \quad \textbf{u} := \textbf{v} - \textbf{w}. \]
Then, $(\textbf{z}, \textbf{u})$ satisfies
\begin{equation*} \label{B-11}
\dot{\textbf{z}} = \textbf{u}, \quad \dot{\textbf{u}}= \kappa_d \textbf{u}.
\end{equation*}
This yields
\begin{align}
\begin{aligned} \label{CD-1-2}
\textbf{u}(t) &= \textbf{v}(t) - \textbf{w}(t) = (\textbf{v}_0 - \textbf{w}_0) e^{\kappa_d t }, \\
\textbf{z}(t) &= \textbf{x}(t) - \textbf{y}(t) =  \frac{\textbf{v}_0 - \textbf{w}_0}{\kappa_d} e^{\kappa_d t} + (\textbf{x}_0 - \textbf{y}_0) - \frac{\textbf{v}_0 - \textbf{w}_0}{\kappa_d}.
\end{aligned}
\end{align}
Hence, as long as $\textbf{v}_0 \not = \textbf{w}_0$, one gets trivial bi-cluster flocking:
\[ \lim_{t \to \infty} |\textbf{u}(t)| = \infty, \quad  \lim_{t \to \infty} |\textbf{z}(t)| = \infty. \]
On the other hand, 
\begin{equation} \label{CD-1-3}
\textbf{v}(t) + \textbf{w}(t) = \textbf{v}_0 + \textbf{w}_0, \quad \textbf{x}(t) + \textbf{y}(t) = \textbf{x}_0 + \textbf{y}_0 + t (\textbf{v}_0 + \textbf{w}_0), \quad t \geq 0.
\end{equation}
Then, it follows from \eqref{CD-1-2} and \eqref{CD-1-3} that 
\begin{align*}
\begin{aligned} 
\textbf{v}(t) &=  \frac{\textbf{v}_0 + \textbf{w}_0}{2} +    \frac{(\textbf{v}_0 - \textbf{w}_0)e^{\kappa_d t}}{2}, \\
\textbf{w}(t) &=   \frac{\textbf{v}_0 + \textbf{w}_0}{2} -    \frac{(\textbf{v}_0 - \textbf{w}_0)e^{\kappa_d t}}{2},  \\
\textbf{x}(t) &= \textbf{x}_0 + \textbf{y}_0 -\frac{\textbf{v}_0 - \textbf{w}_0}{2\kappa_d} + \frac{t}{2} (\textbf{v}_0 + \textbf{w}_0) + \frac{\textbf{v}_0 - \textbf{w}_0}{2\kappa_d} e^{\kappa_d t},     \\
\textbf{y}(t) &= \frac{\textbf{v}_0 - \textbf{w}_0}{2\kappa_d} + \frac{t}{2} (\textbf{v}_0 + \textbf{w}_0) - \frac{\textbf{v}_0 - \textbf{w}_0}{2\kappa_d} e^{\kappa_d t}.
\end{aligned}
\end{align*}
Note that particle $1$ and particle $2$ are completely separated and the velocities are not bounded, and one can also see that global flocking occurs if and only if initial velocities are the same, i.e.,
\[ \textbf{v}_0 = \textbf{w}_0. \]
Next, we consider a three-particle system with system parameters:
\[  (N_1, N_2) = (2,1), \quad \psi_s \equiv 1, \quad \psi_d \equiv 1. \]
In this case, system \eqref{A-1} becomes
\begin{align*}
\begin{aligned} 
\dot{\textbf{x}}_1 & = \textbf{v}_1, \quad \dot{\textbf{x}}_2 = \textbf{v}_2,  \quad \dot{\textbf{y}}_1 = \textbf{w}_1, \\
\dot{\textbf{v}}_1 &= \frac{\kappa_s}{2}  (\textbf{v}_2 - \textbf{v}_1) -\frac{\kappa_d}{2} (\textbf{w}_1 - \textbf{v}_1), \\
\dot{\textbf{v}}_2 &= \frac{\kappa_s}{2}  (\textbf{v}_1 - \textbf{v}_2) -\frac{\kappa_d}{2} (\textbf{w}_1 - \textbf{v}_2) , \\
\dot{\textbf{w}}_1 &= -\frac{\kappa_d}{2}  (\textbf{v}_1 - \textbf{w}_1)  -\frac{\kappa_d}{2}  (\textbf{v}_2 - \textbf{w}_1).
\end{aligned}
\end{align*}
Set
\begin{align*}
\begin{aligned} \label{B-15}
&  \textbf{z}_1 := \textbf{x}_1 - \textbf{x}_2, \quad \textbf{z}_2 := \textbf{x}_2 - \textbf{y}_1, \quad \textbf{z}_1 + \textbf{z}_2 = \textbf{x}_1  - \textbf{y}_1, \\
& \textbf{u}_1 := \textbf{v}_1 - \textbf{v}_2, \quad \textbf{u}_2 := \textbf{v}_2 - \textbf{w}_1, \quad \textbf{u}_1 +\textbf{u}_2 = \textbf{v}_1  - \textbf{w}_1.
\end{aligned}
\end{align*}
Since the velocity dynamics is decoupled from the spatial dynamics, we first consider the velocity dynamics.  Note that the dynamics for $\textbf{u}_1$ and $\textbf{u}_2$ is governed by the following system:
\begin{equation*} \label{B-16}
\dot{\textbf{u}}_1 = \Big(-\kappa_s + \frac{\kappa_d}{2} \Big) \textbf{u}_1, \quad \dot{\textbf{u}}_2 =   \frac{\kappa_s + \kappa_d}{2} \textbf{u}_1 + \frac{3\kappa_d}{2} \textbf{u}_2, \quad t > 0.
\end{equation*}
A direct calculation gives
\begin{align}
\begin{aligned} \label{B-17}
u_1(t) &= u_1(0) e^{-(\kappa_s -\frac{\kappa_d}{2})t}, \quad t \geq 0, \\
u_2(t) &= (u_1(0) + u_2(0)) e^{\frac{3\kappa_d}{2} t} - u_1(0) e^{-(\kappa_s - \frac{\kappa_d}{2})t}.
\end{aligned}
\end{align}
From this explicit formula \eqref{B-17}, it is easy to see that bi-cluster flocking occurs for nontrivial initial data if and only if
\[ \kappa_s > \frac{\kappa_d}{2} > 0. \]

\subsection{A many-body system with $N \geq 4$} \label{sec:3.2} It is easy to see that macro-micro system in Lemma \ref{L3.1} is completely decoupled:
\begin{equation}
\begin{cases} \label{CD-2}
\displaystyle \dot{\textbf{x}}_{c} = \textbf{v}_{c}, \quad \dot{\textbf{y}}_{c} = \textbf{w}_{c}, \quad t > 0,\\
\displaystyle \dot{\textbf{v}}_{c} = \kappa_d (\textbf{v}_{c} - \textbf{w}_{c}), \quad \dot{\textbf{w}}_{c} = \kappa_d (\textbf{w}_{c} - \textbf{v}_{c}),
\end{cases}
\end{equation}
and
\begin{equation}\label{CD-3}
\begin{cases}
\displaystyle \dot{\hat{\textbf{x}}}_{i} = \hat{\textbf{v}}_{i}, \quad \dot{\hat{\textbf{v}}}_{i} =   \frac{\kappa_s}{N_1 }\sum_{k=1}^{N_1}\psi_s(\| \hat{\textbf{x}}_{k} - \hat{\textbf{x}}_{i}\|)\big(\hat{\textbf{v}}_{k} - \hat{\textbf{v}}_{i}\big), \quad t > 0, \\
\displaystyle \dot{\hat{\textbf{y}}}_{j} = \hat{\textbf{w}}_{j}, \quad  \dot{\hat{\textbf{w}}}_{j} =  \frac{\kappa_s}{N_2 }\sum_{k=1}^{N_2}\psi_s(\|\hat{\textbf{y}}_{k} - \hat{\textbf{y}}_{j}\|)\big(\hat{\textbf{w}}_{k} - \hat{\textbf{w}}_{j}\big).
\end{cases}
\end{equation}
Note that the micro-system \eqref{CD-3} is also juxtaposition of two decoupled sub-ensembles. Since the  macro-system is  linear, one can find an explicit dynamics for the average quantities as in the following Lemma.

\begin{lemma} \label{L3.1}
Suppose that system parameters and communication weight functions satisfy \eqref{CD-0}. Then for any solution $\{ (X, V), (Y, W) \}$  to \eqref{CD-1}, one has
\begin{align*}
\begin{aligned}
& \textbf{v}_{c}(t) =  \frac{1}{2}(\textbf{w}_{c}(0) + \textbf{v}_{c}(0)) -  \frac{1}{2} (\textbf{w}_{c}(0) -  \textbf{v}_{c}(0)) e^{2\kappa_d t}, \quad t \geq 0, \\
&  \textbf{w}_{c}(t) =  \frac{1}{2}(\textbf{w}_{c}(0) + \textbf{v}_{c}(0)) +  \frac{1}{2} (\textbf{w}_{c}(0) -  \textbf{v}_{c}(0)) e^{2\kappa_d t}.
\end{aligned}
\end{align*}
\end{lemma}
\begin{proof} It follows from \eqref{CD-2} that 
\[
\frac{d}{dt} ( \textbf{v}_{c} - \textbf{w}_{c}) = 2\kappa_d  ( \textbf{v}_c -  \textbf{w}_c).
\]
This yields
\begin{equation} \label{CD-4}
\textbf{v}_{c}(t) - \textbf{w}_{c}(t) =( \textbf{v}_{c}(0) - \textbf{w}_{c}(0)) e^{2\kappa_d t}.
\end{equation}
On the other hand, 
\begin{equation} \label{CD-5}
\frac{d}{dt} (\textbf{v}_{c}(t) - \textbf{w}_{c}(t)) = 0, \quad \mbox{i.e.,} \quad \textbf{v}_{c}(t) + \textbf{w}_{c}(t) = \textbf{v}_{c}(0) + \textbf{w}_{c}(0).
\end{equation}
Finally, we combine \eqref{CD-4} and \eqref{CD-5} to derive the desired estimates.
\end{proof}
Note that the micro-dynamics for $(\hat{X}, \hat{V})$ and $(\hat{Y}, \hat{W})$ are completely decoupled, and each satisfies the same dynamics for the C-S model. For reader's convenience, we briefly sketch the Lyapunov functional introduced in \cite{H-Liu}.  Introduce the nonlinear functionals:
\[ D(\hat{X}) := \max_{1 \leq i, j \leq N_1} \|\hat{\textbf{x}}_i - \hat{\textbf{x}}_j\|, \qquad D(\hat{V}) :=  \max_{1 \leq i, j \leq N_1} \|\hat{\textbf{v}}_i - \hat{\textbf{v}}_j\|.  \]
Note that
\[ D(X) = D(\hat{X}), \quad  D(V) = D(\hat{V}). \]
Then, these functionals satisfy a system of differential inequality:
\begin{equation}\label{CD-6}
\Big| \frac{d}{dt}D({\hat X}) \Big| \leq D({\hat V}), \quad \frac{d}{dt} D({\hat V}) \leq -\kappa_d \psi_s(D({\hat X})) D({\hat V}), \quad \mbox{a.e.}~ t > 0.
\end{equation}
Now, we introduce Lyapunov-type functionals ${\mathcal L}_{\pm}(t)
\equiv {\mathcal L}_{\pm}(\hat{X}(t), \hat{V}(t))$:
\[
  {\mathcal L}_{\pm}(t) :=  D(\hat{V}) \pm \kappa_d \int_{0}^{D({\hat X})} \psi_s(\eta) d\eta, \quad t \geq 0.
\]
Then, it is easy to see the non-increasing property of ${\mathcal L}_{\pm}$ using \eqref{CD-6}:
\[
{\mathcal L}_{\pm}(t) \leq {\mathcal L}_{\pm}(0), \quad t \geq 0,
\]
which leads to the stability estimate of ${\mathcal L}_{\pm}(t)$:
\begin{equation} \label{CD-7}
D(\hat{V})(t) + \kappa_d  \Big | \int_{D({\hat X}(0))}^{D({\hat X}(t)} \psi_s(\eta)  d\eta \Big | \leq D(\hat{V})(0), \quad t \geq 0.
\end{equation}
The same arguments can be done for $(\hat{Y}, \hat{W})$ to derive a stability estimate:
\begin{equation} \label{CD-8}
D(\hat{W})(t) + \kappa_d  \Big | \int_{D({\hat Y}(0))}^{D({\hat Y}(t)} \psi_s(\eta)  d\eta \Big | \leq D(\hat{W})(0), \quad t \geq 0.
\end{equation}
Then, the stability estimates \eqref{CD-7} and \eqref{CD-8} yield the emergence estimate for the micro-system \eqref{CD-3}.
\begin{theorem} \label{T3.1}
Suppose that the system parameters, the communication weight function and initial data satisfy
\begin{align}
\begin{aligned} \label{CD-9}
& \kappa_s > 0, \quad \kappa_d > 0, \quad D(X(0)) > 0, \quad D(Y(0)) > 0, \quad \| \textbf{v}_{c}(0) - \textbf{w}_{c}(0) \| > 0, \\
& D(V(0)) < \kappa_s \int_{D(X(0))}^{\infty} \psi_s(r) dr, \quad  D(W(0)) < \kappa_s \int_{D(Y(0))}^{\infty} \psi_s(r) dr,
\end{aligned}
\end{align}
and let $\{ (X, V), (Y, W) \}$ be a solution to \eqref{CD-1}. Then, bi-cluster flocking occurs asymptotically in the sense that there exists a positive number $x_M$ and $y_M$ such that
\begin{align*}
\begin{aligned}
& (i)~\sup_{t \geq 0} D(X(t)) \leq x_M, \quad  D(V(t)) \leq D(V(0)) e^{-\kappa_s \psi_s(x_M) t}, \quad t \geq 0, \\
& (ii)~\sup_{t \geq 0} D(Y(t)) \leq x_M, \quad  D(W(t)) \leq D(W(0)) e^{-\kappa_s \psi_s(y_M) t}, \\
& (iii)~\| \textbf{v}_{c}(t) - \textbf{w}_{c}(t) \| = \| \textbf{v}_{c}(0) - \textbf{w}_{c}(0) \| e^{2\kappa_d t}.
\end{aligned}
\end{align*}
\end{theorem}
\begin{proof}
The first two estimates (i) and (ii) follow from the same argument in \cite{H-Liu} and the third assertion also follows from Lemma \ref{L4.1}.
\end{proof}
\begin{remark}
Note that the conditions in \eqref{CD-9} are certainly sufficient conditions. As already studied in \cite{C-H-H-J-K} for the C-S model, once the conditions \eqref{CD-9} are violated,  there can be emerging multi-clusters arising from the initial configuration.
\end{remark}

\section{Localized inter-communication and zero Rayleigh friction} \label{sec:4}
\setcounter{equation}{0}
In this section, we present the emergent flocking dynamics of system \eqref{A-1} with a {\it localized} inter-communication in a priori setting:
\begin{equation*} \label{CD-01}
\psi_s(r) \geq 0, \quad  (\psi_s(r_2) - \psi_s(r_1)) (r_2 - r_1) \leq 0, \quad r_1, r_2 \geq 0, \quad \psi_d(r) \lesssim e^{-\beta r}, \quad \delta = 0,
\end{equation*}
where $\beta$ is a positive constant. \newline

Consider system \eqref{A-1} with $\delta = 0$:
\begin{align}
\begin{aligned} \label{D-1}
\dot{\textbf{x}}_i &= \textbf{v}_i, \quad t > 0, \quad i = 1, \cdots, N_1, \\
\dot{\textbf{v}}_i &= \frac{\kappa_{s}}{N_1} \sum_{k=1}^{N_1} \psi_{s}(\|\textbf{x}_k - \textbf{x}_i\|) (\textbf{v}_k - \textbf{v}_i) -\frac{\kappa_{d}}{N_2} \sum_{k=1}^{N_2} \psi_{d}(\|\textbf{y}_k - \textbf{x}_i\|) (\textbf{w}_k - \textbf{v}_i), \\
\dot{\textbf{y}}_j &= \textbf{w}_j, \quad j = 1, \cdots, N_2, \\
\dot{\textbf{w}}_j &= \frac{\kappa_{s}}{N_2} \sum_{k=1}^{N_2} \psi_{s}(\|\textbf{y}_k - \textbf{y}_j\|) (\textbf{w}_k - \textbf{w}_j) -\frac{\kappa_{d}}{N_1} \sum_{k=1}^{N_1} \psi_{d}(\|\textbf{x}_k - \textbf{y}_j\|) (\textbf{v}_k - \textbf{w}_j).
\end{aligned}
\end{align}
\subsection{A priori estimates on the micro-macro dynamics} \label{sec:4.1} In this subsection, we present some a priori estimates on the evolution of the micro-macro component. Before moving on, we introduce extremal communication weights $\overline{\psi}_s, \underline{\psi}_s, \overline{\psi}_d$ and $\underline{\psi}_d$ as follows. For $t \geq 0$ and $\alpha = s, d,$
\begin{align}
\begin{aligned} \label{D-2}
\overline{\psi_\alpha}(X(t), Y(t)) &:= \max \Big \{ \max_{i,j} \psi_\alpha( \|\textbf{x}_i(t) -   \textbf{x}_j(t) \|),   \max_{i,j} \psi_\alpha( \|\textbf{y}_i(t) -   \textbf{y}_j(t) \|) \Big \}, \\
\underline{\psi_\alpha}(X(t), Y(t)) &:= \min \Big \{ \min_{i,j} \psi_\alpha( \|\textbf{x}_i(t) -   \textbf{x}_j(t) \|),   \min_{i,j} \psi_\alpha( \|\textbf{y}_i(t) -   \textbf{y}_j(t) \|) \Big \}.
\end{aligned}
\end{align}
For notational simplicity, we set
\[ \overline{\psi_\alpha}(t) := \overline{\psi}_\alpha(X(t), Y(t)), \quad \underline{\psi_\alpha}(t) := \underline{\psi_\alpha}(X(t), Y(t)), \quad {\hat M}_2 := M_2( \hat{V} ) + M_2( \hat{W} ), \]
where $M_2(\cdot)$ is a normalized second velocity moment.
\begin{lemma} \label{L4.1}
\emph{(The micro-dynamics)}
Let $\{ (X, V), (Y, W) \}$ be a solution to \eqref{D-1}. Then,  ${\hat M}_2$ satisfies
\[  \frac{d {\hat M}_2}{dt} \leq 2 \Big( \kappa_d \overline{\psi_d} - \kappa_s \underline{\psi_s} \Big) {\hat M}_2 + 2 \kappa_d \overline{\psi_d}  \| \textbf{v}_{c} - \textbf{w}_{c} \| \sqrt{ {\hat M}_2}, \quad t > 0. \]
\end{lemma}
\begin{proof}
(i) It follows from \eqref{C-2} and \eqref{C-3} that
\begin{align}
\begin{aligned} \label{D-3}
\dot{\hat{\textbf{v}}}_{i} = & \frac {\kappa_d}{N_1N_2} \sum_{i=1}^{N_1} \sum_{k=1}^{N_2} \psi_d(\| \hat{\textbf{y}}_{k}
- \hat{\textbf{x}}_{i} \|)\big(\textbf{w}_{k} - \textbf{v}_{i}\big) + \frac{\kappa_s}{N_1 }\sum_{k=1}^{N_1}\psi_s(\|\hat{\textbf{x}}_{k} - \hat{\textbf{x}}_{i}\|)\big(\hat{\textbf{v}}_{k} - \hat{\textbf{v}}_{i}\big) \\
&- \frac{\kappa_d}{N_2} \sum_{k=1}^{N_2}\psi_d(\|\textbf{y}_{k} - \textbf{x}_{i}\|)\big(\textbf{w}_{k} -  \textbf{v}_{i} \big).
\end{aligned}
\end{align}
By taking an inner product $2 \hat{\textbf{v}}_{i}$ with \eqref{D-3} and summing  it over all $i = 1, \cdots, N_1$ using $\displaystyle \sum_{i=1}^{N_1} \hat{\textbf{v}}_i = 0$ and dividing the resulting relation by $N_1$, one obtains
\begin{align}
\begin{aligned} \label{D-4}
&\frac{d}{dt} M_2( \hat{V} ) \\
& = \frac{2\kappa_s}{N^2_1 } \sum_{i, k=1}^{N_1}\psi_s(\|\hat{\textbf{x}}_{k} - \hat{\textbf{x}}_{i}\|)\hat{\textbf{v}}_{i} \cdot \big(\hat{\textbf{v}}_{k} - \hat{\textbf{v}}_{i}\big) - \frac{2\kappa_d}{N_1 N_2}  \sum_{i=1}^{N_1} \sum_{k=1}^{N_2}\psi_d(\|\textbf{y}_{k} - \textbf{x}_{i}\|) \hat{\textbf{v}}_{i} \cdot \big(\textbf{w}_{k} -  \textbf{v}_{i} \big) \\
&\leq -\frac{\kappa_s}{N^2_1 } \sum_{i, k=1}^{N_1}\psi_s(\|\hat{\textbf{x}}_{k} - \hat{\textbf{x}}_{i}\|) \| \hat{\textbf{v}}_{k} - \hat{\textbf{v}}_{i} \|^2 - \frac{2\kappa_d}{N_1 N_2}  \sum_{i=1}^{N_1} \sum_{k=1}^{N_2}\psi_d(\|\textbf{y}_{k} - \textbf{x}_{i}\|) \hat{\textbf{v}}_{i} \cdot \big(\textbf{w}_{k} -  \textbf{v}_{i} \big),
\end{aligned}
\end{align}
where we used the standard index interchange trick. Similarly, one has
\begin{align}
\begin{aligned} \label{D-5}
\frac{d}{dt} M_2( \hat{W} ) \leq & -\frac{\kappa_s}{N^2_2} \sum_{j, k=1}^{N_2}\psi_s(\|\hat{\textbf{y}}_{k} - \hat{\textbf{y}}_{j}\|) \| \hat{\textbf{w}}_{k} - \hat{\textbf{w}}_{j} \|^2 \\
&- \frac{2\kappa_d}{N_1N_2 }  \sum_{j=1}^{N_2} \sum_{k=1}^{N_1}\psi_d(\|\textbf{x}_{k} - \textbf{y}_{j}\|) \hat{\textbf{w}}_{j} \cdot \big(\textbf{v}_{k} -  \textbf{w}_{j} \big).
\end{aligned}
\end{align}
Then, combining \eqref{D-4}, \eqref{D-5} and using the micro-macro decomposition:
\[ \textbf{v}_{i} - \textbf{w}_{j} = \hat{\textbf{v}}_{i} - \hat{\textbf{w}}_{j} + \textbf{v}_c - \textbf{w}_c \]
gives
\begin{align}
\begin{aligned} \label{D-6}
\frac{d}{dt} {\hat M}_2 \leq & -\frac{\kappa_s}{N^2_1 } \sum_{i, k=1}^{N_1}\psi_s(\|\hat{\textbf{x}}_{k} - \hat{\textbf{x}}_{i}\|) \| \hat{\textbf{v}}_{k} - \hat{\textbf{v}}_{i} \|^2 -\frac{\kappa_s}{N^2_2} \sum_{j, k=1}^{N_2}\psi_s(\|\hat{\textbf{y}}_{k} - \hat{\textbf{y}}_{j}\|) \| \hat{\textbf{w}}_{k} - \hat{\textbf{w}}_{j} \|^2 \\
& + \frac{2\kappa_d}{N_1 N_2}  \sum_{i=1}^{N_1} \sum_{j=1}^{N_2}\psi_d(\|\textbf{y}_{j} - \textbf{x}_{i}\|) \| \hat{\textbf{v}}_{i} - \hat{\textbf{w}}_{j} \|^2  \\
& + \frac{2\kappa_d}{N_1 N_2}  \sum_{i=1}^{N_1} \sum_{j=1}^{N_2}\psi_d(\|\textbf{y}_{j} - \textbf{x}_{i}\|) (\hat{\textbf{v}}_{i} - \hat{\textbf{w}}_{j} ) \cdot \big(\textbf{v}_{c} - \textbf{w}_{c} \big) \\
=: &{\mathcal I}_{11} + {\mathcal I}_{12} + {\mathcal I}_{13}  + {\mathcal I}_{14}.
\end{aligned}
\end{align}
Below, we estimate the terms ${\mathcal I}_{1i}$ separately. \newline

\noindent $\bullet$ (Estimate of ${\mathcal I}_{11}$): We use
\[ \psi_s(\|\textbf{x}_{k} - \textbf{x}_{i}\|) \geq \underline{\psi_s}, \quad \sum_{i, k=1}^{N_1} \| \hat{\textbf{v}}_{i} - \hat{\textbf{v}}_{k}  \|^2 =  2 N_1^2 M_2({\hat V})   \]
to get
\begin{equation} \label{D-7}
{\mathcal I}_{11} =  -\frac{\kappa_s}{N^2_1 } \sum_{i, k=1}^{N_1}\psi_s(\|\textbf{x}_{k} - \textbf{x}_{i}\|) \| \hat{\textbf{v}}_{k} - \hat{\textbf{v}}_{i} \|^2 \leq -2 \kappa_s  \underline{\psi_s} M_2({\hat V}).
\end{equation}

\vspace{0.2cm}

\noindent $\bullet$ (Estimate of ${\mathcal I}_{12}$): Similarly, one also has
\begin{equation} \label{D-8}
{\mathcal I}_{12}  \leq -2 \kappa_s  \underline{\psi_s} M_2({\hat W}).
\end{equation}

\vspace{0.2cm}

\noindent $\bullet$ (Estimate of ${\mathcal I}_{13}$): We use
\[ \psi_d(\|\textbf{x}_{k} - \textbf{x}_{i}\|) \leq \overline{\psi_d} \quad \mbox{and} \quad \sum_{i=1}^{N_1} \sum_{j=1}^{N_2} \| \hat{\textbf{v}}_{i} - \hat{\textbf{w}}_{j}  \|^2 =   N_1 N_2  \hat{M}_2 \]
to get
\begin{equation} \label{D-9}
{\mathcal I}_{13} = \frac{2\kappa_d}{N_1 N_2}  \sum_{i=1}^{N_1} \sum_{j=1}^{N_2}\psi_d(\|\textbf{y}_{j} - \textbf{x}_{i}\|) \| \hat{\textbf{v}}_{i} - \hat{\textbf{w}}_{j} \|^2 \leq 2 \kappa_d \overline{\psi_d} \hat{M}_2.
\end{equation}

\vspace{0.2cm}

\noindent $\bullet$ (Estimate of ${\mathcal I}_{14}$): By the Cauchy-Schwarz inequality, one gets

\begin{equation} \label{D-10}
 {\mathcal I}_{14} \leq \frac{2\kappa_d \overline{\psi_d} }{N_1 N_2}  \sum_{i=1}^{N_1} \sum_{j=1}^{N_2} \| \hat{\textbf{v}}_{i} - \hat{\textbf{w}}_{j} \| \cdot \| \textbf{v}_{c} - \textbf{w}_{c} \| \leq 2\kappa_d \overline{\psi_d}  \| \textbf{v}_{c} - \textbf{w}_{c} \| \sqrt{\hat{M}_2}.
\end{equation}
In \eqref{D-6}, we combine all estimates \eqref{D-7}, \eqref{D-8}, \eqref{D-9} and \eqref{D-10} to obtain the desired estimate.
\end{proof}

\begin{lemma} \label{L4.2}
\emph{(The macro-dynamics)}
Let $\{ (X, V), (Y, W) \}$ be a solution to \eqref{D-1}. Then, $ \| \textbf{v}_{c} - \textbf{w}_{c} \|$ satisfies
\begin{eqnarray*}
&& (i)~\frac{d}{dt} \| \textbf{v}_{c} - \textbf{w}_{c} \| \leq 2\kappa_d \overline{\psi_d}\| \textbf{v}_c -  \textbf{w}_c \| +  2\kappa_d \overline{\psi_d} \sqrt{{\hat M}_2}, \quad t > 0, \cr
&& (ii)~\frac{d}{dt} \| \textbf{v}_{c} - \textbf{w}_{c} \| \geq 2\kappa_d \underline{\psi_d}\| \textbf{v}_c -  \textbf{w}_c \| - 2\kappa_d \overline{\psi_d} \sqrt{{\hat M}_2}.
\end{eqnarray*}
\end{lemma}
\begin{proof} $\bullet$~(The first differential inequality): It follows from \eqref{C-2} that
\begin{align}
\begin{aligned} \label{D-11}
\frac{d}{dt} \Big( \textbf{v}_{c} - \textbf{w}_{c} \Big) &= \frac {2\kappa_d}{N_1N_2} \sum_{j=1}^{N_2} \sum_{i=1}^{N_1} \psi_d(\|\textbf{x}_{i}
- \textbf{y}_{j} \|) ( \textbf{v}_c -  \textbf{w}_c) +  \frac {2\kappa_d}{N_1N_2} \sum_{i=1}^{N_1} \sum_{j=1}^{N_2} \psi_d(\|\textbf{x}_{i}
- \textbf{y}_{j} \|)\big(\hat{\textbf{v}}_{i} - \hat{\textbf{w}}_{j} \big).
\end{aligned}
\end{align}
By taking an inner product $2( \textbf{v}_{c} - \textbf{w}_{c})$ with \eqref{D-11}, one obtains
\begin{align*}
\begin{aligned}
\frac{d}{dt} \| \textbf{v}_{c} - \textbf{w}_{c} \|^2 &= \frac {4\kappa_d}{N_1N_2} \sum_{j=1}^{N_2} \sum_{i=1}^{N_1} \psi_d(\|\textbf{x}_{i}
- \textbf{y}_{j} \|)   \| \textbf{v}_c -  \textbf{w}_c \|^2 \\
&+  \frac {4\kappa_d}{N_1N_2} \sum_{i=1}^{N_1} \sum_{j=1}^{N_2} \psi_d(\|\textbf{y}_{j}
- \textbf{x}_{i} \|)  ( \textbf{v}_{c} - \textbf{w}_{c} ) \cdot \big(\hat{\textbf{v}}_{i} - \hat{\textbf{w}}_{j} \big).
\end{aligned}
\end{align*}
We again use the same argument in \eqref{D-9} to get
\begin{equation} \label{D-11-1}
\frac{d}{dt} \| \textbf{v}_{c} - \textbf{w}_{c} \|^2 \leq 4\kappa_d \overline{\psi_d}(t)  \| \textbf{v}_c -  \textbf{w}_c \|^2 +  4\kappa_d \overline{\psi_d}  \| \textbf{v}_{c} - \textbf{w}_{c} \| \sqrt{{\hat M}_2}.
\end{equation}
This implies the desired first differential inequality. \newline

\noindent $\bullet$~(The second differential inequality): It follows from \eqref{D-11} that 
\begin{equation} \label{D-12}
\frac{d}{dt} \| \textbf{v}_{c} - \textbf{w}_{c} \|^2 \geq 4\kappa_d \underline{\psi_d} \| \textbf{v}_c -  \textbf{w}_c \|^2  -4\kappa_d \overline{\psi_d}  \| \textbf{v}_{c} - \textbf{w}_{c} \| \sqrt{{\hat M}_2}.
\end{equation}
Hence, \eqref{D-11-1} and \eqref{D-12} yield the desired estimates.
\end{proof}

\begin{lemma} \label{L4.3}
Let $\{ (X, V), (Y, W) \}$ be a solution to \eqref{D-1}. Then, 
\[ \|\textbf{v}_c \| \leq \sqrt{M_2(V)}, \quad   \|\textbf{w}_c \| \leq \sqrt{M_2(W)}, \quad  M_2(t) \leq  M_2(0) e^{4 \kappa_d \psi_d^{\infty} t}. \]
\end{lemma}
\begin{proof}(i)~ By the Cauchy-Schwarz inequality, 
\[  \|\textbf{v}_c \| \leq \frac{1}{N_1} \sum_{i=1}^{N_1} \| \textbf{v}_i \|  \leq   \frac{1}{\sqrt{N_1}} \Big( \sum_{i=1}^{N_1} \| \textbf{v}_i \|^2  \Big)^{\frac{1}{2}} = \sqrt{M_2(V)}.  \]
The other estimate can be done exactly the same way. \newline

\noindent (ii)~It follows from Lemma \ref{L2.1} that 
\begin{align*}
\begin{aligned}
\frac{dM_2}{dt}  =  &- \kappa_s \Big[ \frac{1}{N^2_1}  \sum_{i, k=1}^{N_1} \psi_{s}(|\hat{\textbf{x}}_k - \hat{\textbf{x}}_i|) |\hat{\textbf{v}}_k - \hat{\textbf{v}}_i|^2 + \frac{1}{N^2_2}  \sum_{j, k=1}^{N_2} \psi_{s}(|\hat{\textbf{y}}_k - \hat{\textbf{y}}_j|) |\hat{\textbf{w}}_k - \hat{\textbf{w}}_j|^2 \Big] \\
&  + \frac{2\kappa_{d}}{N_1 N_2} \sum_{i=1}^{N_1} \sum_{j=1}^{N_2}  \psi_{d}(|\textbf{y}_j - \textbf{x}_i|) |\textbf{v}_i - \textbf{w}_j|^2 \leq 4 \kappa_d \psi_d^{\infty} M_2.
\end{aligned}
\end{align*}
This yields the desired estimate.
\end{proof}

\subsection{Emergence of bi-cluster flocking} \label{sec:4.2} As a direct application of Lemma \ref{L4.1} and Lemma \ref{L4.2}, we have the following bi-cluster flocking estimate.
\begin{theorem} \label{T4.1}
Suppose that the initial velocity and the following a priori conditions hold: there exist positive constants $C_0, \varepsilon_0, {\tilde \varepsilon}_0, \eta_0$ and ${\tilde \eta}_0$ such that
\begin{align}
\begin{aligned} \label{D-12-1}
& |\overline{\psi_d}(t) e^{2\kappa_d \psi_d^{\infty}t} | \leq C_0 e^{-\varepsilon_0 \kappa_d \psi_d^{\infty} t}, \quad \underline{\psi_d}(t) \geq {\tilde C}_0 e^{-{\tilde \varepsilon}_0
\kappa_d \psi_d^{\infty} t}, \quad t \geq 0, \\
&  \| \textbf{v}_{c}(0) - \textbf{w}_{c}(0) \|  \gg 1, \quad \sup_{0 \leq t < \infty} \Big( \kappa_d \overline{\psi_d}(t) - \kappa_s \underline{\psi_s}(t) \Big) \leq -\frac{\eta_0}{2},
\end{aligned}
\end{align}
then, we have a bi-cluster flocking in the sense of Definition \ref{D1.1}.
\end{theorem}
\begin{proof} Note that Lemma \ref{L4.1}, Lemma \ref{L4.2} and \eqref{D-12-1} imply
\begin{align}
\begin{aligned} \label{D-13}
& \frac{d {\hat M}_2}{dt} \leq -\eta_0 {\hat M}_2 + 2 \kappa_d \overline{\psi_d}  \| \textbf{v}_{c} - \textbf{w}_{c} \| \sqrt{ {\hat M}_2}, \quad t > 0, \\
& \frac{d}{dt} \| \textbf{v}_{c} - \textbf{w}_{c} \| \leq 2\kappa_d \overline{\psi_d}\| \textbf{v}_c -  \textbf{w}_c \| +  2\kappa_d \overline{\psi_d} \sqrt{{\hat M}_2}, \\
& \frac{d}{dt} \| \textbf{v}_{c} - \textbf{w}_{c} \| \geq 2\kappa_d \underline{\psi_d}\| \textbf{v}_c -  \textbf{w}_c \| - 2\kappa_d \overline{\psi_d} \sqrt{{\hat M}_2}.
\end{aligned}
\end{align}
Below, we present a bi-cluster flocking estimate in several steps.  \newline

\noindent $\bullet$ Step A (The exponential bound for $\| \textbf{v}_{c} - \textbf{w}_{c} \|$): It follows from Lemma \ref{L4.3} that
\begin{align}
\begin{aligned} \label{D-14}
\| \textbf{v}_{c} - \textbf{w}_{c} \| &\leq  \| \textbf{v}_{c} \| + \| \textbf{w}_{c} \| \leq \sqrt{M_2(V)} + \sqrt{M_2(W)} \\
& \leq \sqrt{2M_2} \leq \sqrt{2 M_2(0)} e^{2\kappa_d \psi_d^{\infty} t}.
\end{aligned}
\end{align}

\vspace{0.2cm}

\noindent $\bullet$ Step B (The exponential decay of ${\hat M}_2$): We substitute the estimate \eqref{D-14} into $\eqref{D-13}_1$ and use a priori condition $\eqref{D-12-1}_1$
 to obtain
\begin{align*}
\begin{aligned}
\frac{d {\hat M}_2}{dt} &\leq -\eta_0 {\hat M}_2 + 2  \sqrt{2 M_2(0)} \kappa_d \overline{\psi_d}   e^{2\kappa_d \psi_d^{\infty} t} \sqrt{ {\hat M}_2} \\
&\leq -\eta_0 {\hat M}_2 + 2 C_0 \sqrt{2 M_2(0)} \kappa_d  e^{-\varepsilon_0 \kappa_d \psi_d^{\infty} t}  \sqrt{ {\hat M}_2}.
\end{aligned}
\end{align*}
To convert the above differential inequality into a familiar Gronwall's differential inequality, set
\[  {\mathcal M}(t) := \sqrt{ {\hat M}_2(t)}  \]
Then, it is easy to see that ${\mathcal M}$ satisfies
\[  \frac{d{\mathcal M}}{dt} \leq -\frac{\eta_0}{2}  {\mathcal M} + C_0 \sqrt{2 {\hat M}_2(0)} \kappa_d e^{-\varepsilon_0 \kappa_d \psi_d^{\infty} t}.  \]
We now apply Lemma \ref{LA.2} in Appendix A with the choices:
\[ \alpha = \frac{\eta_0}{2}, \quad f = C_0 \sqrt{2 {\hat M}_2(0)} \kappa_d e^{-\varepsilon_0 \kappa_d \psi_d^{\infty} t} \]
to get
\[ {\mathcal M}(t)  \leq \frac{2C_0 \sqrt{2 {\hat M}_2(0)} \kappa_d }{\eta_0}  e^{-\frac{\varepsilon_0 \kappa_d \psi_d^{\infty}}{2} t}  +  \sqrt{ {\hat M}_2(0)} e^{-\frac{\eta_0}{2}t} +
 \frac{2C_0 \sqrt{2 {\hat M}_2(0)} \kappa_d }{\eta_0} e^{-\frac{\eta_0}{4}t}.  \]
We set
 \[ C_1 := \max \Big\{ \frac{2C_0 \sqrt{2 {\hat M}_2(0)} \kappa_d }{\eta_0},~ \sqrt{ {\hat M}_2(0)} \Big \}. \]
 Then, 
\begin{equation} \label{D-15}
 \sqrt{ {\hat M}_2(t)} = {\mathcal M}(t)  \leq 3 C_1 \max \{ e^{-\frac{\varepsilon_0 \kappa_d \psi_d^{\infty}}{2} t},~e^{-\frac{\eta_0}{4}t} \}.
\end{equation}

\vspace{0.2cm}

\noindent $\bullet$ Step C (Uniform boundness of $\| \textbf{v}_{c} - \textbf{w}_{c} \|$): In $\eqref{D-13}_2$, by using \eqref{D-15} and $\eqref{D-12-1}_1$ one gets
\begin{align}
\begin{aligned}  \label{D-16}
\frac{d}{dt} \| \textbf{v}_{c} - \textbf{w}_{c} \| &\leq 2\kappa_d C_0 e^{-(2 + \varepsilon_0 )\kappa_d \psi_d^{\infty} t}
\| \textbf{v}_c -  \textbf{w}_c \| \\
&+  6C_0 C_1 \kappa_d e^{-(2 + \varepsilon_0 )\kappa_d \psi_d^{\infty} t} \max \{ e^{-\frac{\varepsilon_0 \kappa_d \psi_d^{\infty}}{2} t},~e^{-\frac{\eta_0}{4}t} \}.
\end{aligned}
\end{align}
We again apply Lemma \ref{LA.3} for \eqref{D-16} with
\begin{align*}
\begin{aligned}
\alpha(t) &:= 2\kappa_d C_0 e^{-(2 + \varepsilon_0 )\kappa_d \psi_d^{\infty} t}, \\
 f(t) &:=  6C_0 C_1 \kappa_d e^{-(2 + \varepsilon_0 )\kappa_d \psi_d^{\infty} t} \max \{ e^{-\frac{\varepsilon_0 \kappa_d \psi_d^{\infty}}{2} t},~e^{-\frac{\eta_0}{4}t} \} \leq  6C_0 C_1 \kappa_d e^{-(2 + \varepsilon_0 )\kappa_d \psi_d^{\infty} t}
\end{aligned}
\end{align*}
to obtain
\begin{equation} \label{D-17}
\| \textbf{v}_{c}(t) - \textbf{w}_{c}(t) \| \leq \Big(  \| \textbf{v}_{c}(0) - \textbf{w}_{c}(0) \| +  \frac{6C_0 C_1}{(2 + \varepsilon_0) \psi_d^{\infty}} \Big) e^{ \frac{2C_0}{(2 + \varepsilon_0) \psi_d^{\infty}}} =: C_2.
\end{equation}

\vspace{0.2cm}

\noindent $\bullet$ Step D (Improving the decay estimate of ${\hat M}_2$): We use \eqref{D-17} and $\eqref{D-13}_1$ to get
\[  \frac{d {\hat M}_2}{dt} \leq -\eta_0 {\hat M}_2 + 2 \kappa_d C_0 C_2 e^{-(2 + \varepsilon_0) \kappa_d \psi_d^{\infty} t} \sqrt{ {\hat M}_2}. \]
By the same argument as in Step B, 
\[ \frac{d{\mathcal M}}{dt} \leq -\frac{\eta_0}{2} {\mathcal M} + \kappa_d C_0 C_2 e^{-(2 + \varepsilon_0) \kappa_d \psi_d^{\infty} t}. \]
Then, 
\begin{equation*} \label{D-18}
  \sqrt{ {\hat M}_2(t)} = {\mathcal M}(t) \leq 3 C_3 \max \Big \{ e^{-\frac{(2 + \varepsilon_0) \kappa_d \psi_d^{\infty}}{2} t},~e^{-\frac{\eta_0}{4}t} \Big \}. \
\end{equation*}
With this improved estimate, we can also slightly refine the upper bound estimate for $\| \textbf{v}_{c}(t) - \textbf{w}_{c}(t) \| $.

\vspace{0.2cm}

\noindent $\bullet$ Step E (Spatial boundedness of each subensemble):  Since ${\hat M}_2(t)$ decays to zero exponentially, it is easy to see that
\begin{align*}
\begin{aligned}
& \sup_{0 \leq t < \infty} \max_{i} |\textbf{x}_i(t) - \textbf{x}_c(t) | \leq  \sup_{0 \leq t < \infty} \max_{i} |\textbf{x}_i(0) - \textbf{x}_c(0) |  + C(N, \varepsilon_0, \kappa_d, \psi_d^{\infty}, \eta_0), \cr
& \sup_{0 \leq t < \infty} \max_{i} |\textbf{y}_i(t) - \textbf{y}_c(t) | \leq  \sup_{0 \leq t < \infty} \max_{i} |\textbf{y}_i(0) - \textbf{y}_c(0) |  + C(N, \varepsilon_0, \kappa_d, \psi_d^{\infty}, \eta_0).
\end{aligned}
\end{align*}

\vspace{0.2cm}

\noindent  $\bullet$ Step F (Separation of two sub-ensembles): We use $\eqref{D-12-1}_1$ and $\eqref{D-13}_3$ to find
\begin{align*}
\begin{aligned}
\frac{d}{dt} \| \textbf{v}_{c} - \textbf{w}_{c} \| &\geq 2\kappa_d \underline{\psi_d}\| \textbf{v}_c -  \textbf{w}_c \| - 2\kappa_d \overline{\psi_d} \sqrt{{\hat M}_2} \\
& \geq 2 \kappa_d {\tilde C}_0 e^{-{\tilde \varepsilon}_0 \kappa_d \psi_d^{\infty} t} \| \textbf{v}_c -  \textbf{w}_c \| - 6 \kappa_d C_0 C_3 \max \{ e^{-\frac{3(2 + \varepsilon_0) \kappa_d \psi_d^{\infty}}{2} t},~e^{-(\frac{\eta_0}{4} + (\varepsilon_0 + 2) \kappa_d \psi_d^{\infty}) t} \}.
\end{aligned}
\end{align*}
We set
\[ \alpha :=  2 \kappa_d {\tilde C}_0 e^{-{\tilde \varepsilon}_0 \kappa_d \psi_d^{\infty} t}, \quad f(t) :=  6 \kappa_d C_0 C_3 \max \{ e^{-\frac{3(2 + \varepsilon_0) \kappa_d \psi_d^{\infty}}{2} t},~e^{-(\frac{\eta_0}{4} + (\varepsilon_0 + 2) \kappa_d \psi_d^{\infty}) t} \},   \]
Then, 
\[ \|\alpha\|_{L^1} = \frac{2 {\tilde C}_0}{{\tilde \varepsilon}_0 \psi_d^{\infty}}, \qquad ||f ||_{L^1} \leq \max \Big \{ \frac{4C_0 C_3}{(2 + \varepsilon_0) \psi_d^{\infty}},~
\frac{6 \kappa_d C_0 C_3}{ \frac{\eta_0}{4} + (\varepsilon_0 + 2) \kappa_d \psi_d^{\infty}} \Big \} =: C_4.    \]
We again apply Lemma \ref{LA.4} to get the separation of two groups:
\[  \| \textbf{v}_{c}(t) - \textbf{w}_{c}(t) \| \geq \| \textbf{v}_{c}(0) - \textbf{w}_{c}(0) \| - C_4 e^{ \frac{2 {\tilde C}_0}{{\tilde \varepsilon}_0 \psi_d^{\infty}}} =: C_5 > 0.   \]
\end{proof}
\begin{remark}
Note that a constant inter communication weight $\psi_d \equiv 1$ does not satisfy a priori condition $\eqref{D-12-1}_1$. Thus, the framework in Theorem \ref{T4.1} does not cover the framework in Section \ref{sec:3}.  In contrast, for the choices:
\[ \psi_s(r) \equiv 1, \quad \psi_d(r) = e^{-\beta \kappa_d \psi_d^{\infty} r}, \quad \beta \gg 1,  \quad \kappa_d \ll \kappa_s \]
it is easy to see that a priori conditions \eqref{D-12-1} hold.
\end{remark}
\section{Emergence of bi-cluster flocking with the Rayleigh friction} \label{sec:5}
\setcounter{equation}{0}
In this section, we study the emergence of bi-cluster flocking from spatially mixed configurations to the bi-cluster flocking configuration under the effect of Rayleigh friction.  Similar to the previous section, we take the following two steps:
\begin{itemize}
\item
Step A (Decay of the local fluctuations): We show that $M_2$ decays to zero as $t \to \infty$:
\begin{equation} \label{E-0-1}
\lim_{t \to \infty} {\hat M_2}(t) = 0.
\end{equation}
\item
Step B (Emergence of separation): We show that the relative difference of local average velocities is bounded away from zero uniformly in time:
\begin{equation*} \label{E-0-2}
   \inf_{0 \leq t < \infty} \| \textbf{v}_c(t) -  \textbf{w}_c(t) \| > 0.
\end{equation*}
\end{itemize}

\subsection{Evolution of the micro dynamics} \label{sec:5.1}
Next, we establish the estimate \eqref{E-0-1} by focusing on the effect of the Rayleigh friction.
\begin{lemma} \label{L5.1}
Let $\{ (X, V), (Y, W) \}$ be a solution to \eqref{A-1} with $\delta > 0$. Then, $ {\hat M}_2$ satisfies
\[  \frac{d {\hat M_2}}{dt} \leq 2 \Big( \kappa_d \overline{\psi_d} - \kappa_s \underline{\psi_s} +\delta \Big) {\hat M_2} + 2 \kappa_d \overline{\psi_d}  \| \textbf{v}_{c} - \textbf{w}_{c} \| \sqrt{ {\hat M_2}}, \quad t > 0, \]
where $\overline{\psi_d} $ and $\underline{\psi_s}$ are extremal functions in \eqref{D-2} depending on the solution itself.
\end{lemma}
\begin{proof}
Since the estimates will be similar to that in Lemma \ref{L4.1} except the terms involving with the  Rayleigh friction, we mainly focus on the terms involved with
the Rayleigh friction. \newline

\noindent (i) It follows from \eqref{C-2} and \eqref{C-3} that
\begin{align}
\begin{aligned} \label{E-1}
\dot{\hat{\textbf{v}}}_{i}  = & \frac{\kappa_s}{N_1 }\sum_{k=1}^{N_1}\psi_s(\|\hat{\textbf{x}}_{k} - \hat{\textbf{x}}_{i}\|)\big(\hat{\textbf{v}}_{k} - \hat{\textbf{v}}_{i}\big) - \frac{\kappa_d}{N_2} \sum_{k=1}^{N_2}\psi_d(\|\textbf{y}_{k} - \textbf{x}_{i}\|)\big(\textbf{w}_{k} -  \textbf{v}_{i} \big) \\
& + \delta \textbf{v}_{i} (1-\|\textbf{v}_{i}\|^2) - \dot{\textbf{v}}_c.
\end{aligned}
\end{align}
We take the inner product $2\hat{\textbf{v}}_i$ with \eqref{E-1}, sum it over all $i =1, \cdots, N_1$ using $\sum_{i} \hat{\textbf{v}}_{i} = 0$ and symmetry trick to get
\begin{align}
\begin{aligned} \label{E-2}
&\frac{d}{dt} M_2 ( \hat{V} ) = -\frac{\kappa_s}{N^2_1 }\sum_{i, k=1}^{N_1}\psi_s(\|\hat{\textbf{x}}_{k} - \hat{\textbf{x}}_{i}\|) \| \hat{\textbf{v}}_{k} - \hat{\textbf{v}}_{i} \|^2 \\
&- \frac{2\kappa_d}{N_1 N_2} \sum_{i=1}^{N_1} \sum_{k=1}^{N_2}\psi_d(\|\textbf{y}_{k} - \textbf{x}_{i}\|) \big(\textbf{w}_{c} -  \textbf{v}_{c} + \hat{\textbf{w}}_{k} -  \hat{\textbf{v}}_{i} \big)  \cdot \hat{\textbf{v}}_{i} + \frac{2 \delta}{N_1} \sum_{i=1}^{N_1} \hat{\textbf{v}}_{i}  \cdot \textbf{v}_{i}  (1-\|\textbf{v}_{i}\|^2).
\end{aligned}
\end{align}
Similarly, 
\begin{align}
\begin{aligned} \label{E-3}
&\frac{d}{dt} {\hat M_2}( \hat{W} )  =   -\frac{\kappa_s}{N^2_2} \sum_{j, k=1}^{N_2}\psi_s(\|\hat{\textbf{y}}_{k} - \hat{\textbf{y}}_{j}\|) \| \hat{\textbf{w}}_{k} - \hat{\textbf{w}}_{i} \|^2 \\
& - \frac{2\kappa_d}{N_1 N_2}  \sum_{k=1}^{N_1} \sum_{j=1}^{N_2}\psi_d(\|\textbf{y}_{j} - \textbf{x}_{k}\|)( \textbf{v}_{c} - \textbf{w}_{c} +
 \hat{\textbf{v}}_{k}  - \hat{\textbf{w}}_{j} )  \cdot  \hat{\textbf{w}}_{j} + \frac{2\delta}{N_2}  \sum_{j=1}^{N_2}  \hat{\textbf{w}}_j  \cdot \textbf{w}_{j}  (1-\|\textbf{w}_{j}\|^2).
\end{aligned}
\end{align}
Now, combining \eqref{E-2} and \eqref{E-3} yields
\begin{align}
\begin{aligned} \label{E-4}
\frac{d}{dt} {\hat M_2}  =  &-\frac{\kappa_s}{N^2_1 } \sum_{i, k=1}^{N_1}\psi_s(\|\hat{\textbf{x}}_{k} -\hat{\textbf{x}}_{i}\|) \| \hat{\textbf{v}}_{k} - \hat{\textbf{v}}_{i} \|^2 -\frac{\kappa_s}{N^2_2} \sum_{j, k=1}^{N_2}\psi_s(\|\hat{\textbf{y}}_{k} - \hat{\textbf{y}}_{j}\|) \| \hat{\textbf{w}}_{k} - \hat{\textbf{w}}_{j} \|^2 \\
& + \frac{2\kappa_d}{N_1 N_2}  \sum_{i=1}^{N_1} \sum_{j=1}^{N_2}\psi_d(\|\textbf{y}_{j} - \textbf{x}_{i}\|) \| \hat{\textbf{v}}_{i} - \hat{\textbf{w}}_{j} \|^2  \\
& + \frac{2\kappa_d}{N_1 N_2}  \sum_{i=1}^{N_1} \sum_{j=1}^{N_2}\psi_d(\|\textbf{y}_{j} - \textbf{x}_{i}\|) (\hat{\textbf{v}}_{i} - \hat{\textbf{w}}_{j} ) \cdot \big(\textbf{v}_{c} - \textbf{w}_{c} \big) \\
& +  \frac{2\delta}{N_1} \sum_{i=1}^{N_1} (1-\|\textbf{v}_{i}\|^2)  \textbf{v}_{i} \cdot \hat{\textbf{v}}_i
+ \frac{2\delta}{N_2}  \sum_{j=1}^{N_2}| (1-\|\textbf{w}_{j}\|^2)  \textbf{w}_{j} \cdot \hat{\textbf{w}}_j  \\
 =: &\sum_{l=1}^{6} {\mathcal I}_{2l}.
\end{aligned}
\end{align}
Below, we estimate the terms ${\mathcal I}_{2l},~l=1, \cdots, 6$ as follows. \newline

\noindent $\bullet$ (Estimates of ${\mathcal I}_{2l},~l=1,2,3,4$): By the same arguments as in Section \ref{sec:4} one gets
\begin{align}
\begin{aligned} \label{E-4-1}
& {\mathcal I}_{21}  \leq   -2\kappa_s \underline{\psi_s}  M_2({\hat V}), \quad {\mathcal I}_{22}  \leq   -2\kappa_s \underline{\psi_s}  M_2({\hat W}), \\
& {\mathcal I}_{23} \leq  2\kappa_d \overline{\psi_d}{\hat M_2}, \quad {\mathcal I}_{24} \leq 2 \kappa_d \overline{\psi_d} \| \textbf{v}_{c} - \textbf{w}_{c} \| \cdot \sqrt{{\hat M}_2}.
\end{aligned}
\end{align}

\noindent $\bullet$ (Estimate of ${\mathcal I}_{2l},~l=5,6$): Using $ \textbf{v}_{i} =  \hat{\textbf{v}}_{i} +  \textbf{v}_{c}$ and  $\sum_{i=1}^{N_1}  \hat{\textbf{v}}_i   = 0$ gives 
\begin{align}
\begin{aligned} \label{E-5}
{\mathcal I}_{25} &=  \frac{2\delta}{N_1} \sum_{i=1}^{N_1} (1-\|\textbf{v}_{i}\|^2) \| \hat{\textbf{v}}_i \|^2 + \frac{2\delta}{N_1} \sum_{i=1}^{N_1} (1-\|\textbf{v}_{i}\|^2)  \textbf{v}_{c} \cdot \hat{\textbf{v}}_i   \\
&= 2\delta M_2(\hat V )
- \frac{2\delta}{N_1} \sum_{i=1}^{N_1} \|\textbf{v}_{i}\|^2 \| \hat{\textbf{v}}_i \|^2
- \frac{2\delta}{N_1} \sum_{i=1}^{N_1} \|\textbf{v}_{i}\|^2  \textbf{v}_{c} \cdot \hat{\textbf{v}}_i \\
&=  2\delta M_2(\hat V )
- \frac{2\delta}{N_1} \sum_{i=1}^{N_1} \|\textbf{v}_{i}\|^2   \textbf{v}_i \cdot \hat{\textbf{v}}_i .
\end{aligned}
\end{align}
Now, we further estimate the second term in \eqref{E-5} as follows.
\begin{align}
\begin{aligned} \label{E-6}
&\sum_{i=1}^{N_1} \|\textbf{v}_{i}\|^2   \textbf{v}_i \cdot \hat{\textbf{v}}_i  = \sum_{i=1}^{N_1} \|\textbf{v}_{i}\|^2  ( \| \textbf{v}_i \|^2 -  \textbf{v}_c\cdot \hat{\textbf{v}}_i ) \\
& \hspace{0.5cm} =  \frac{1}{2}\sum_{i=1}^{N_1} \|\textbf{v}_{i}\|^2  ( \| \textbf{v}_i \|^2 - \| \textbf{v}_c \|^2  + \| \textbf{v}_i - \textbf{v}_c \|^2 )
\geq  \frac{1}{2}\sum_{i=1}^{N_1} \|\textbf{v}_{i}\|^2  ( \| \textbf{v}_i \|^2 - \| \textbf{v}_c \|^2  ) \\
&  \hspace{0.5cm} =  \frac{1}{2}\sum_{i=1}^{N_1} \|\textbf{v}_{i}\|^4  -  \frac{1}{2}\sum_{i=1}^{N_1} \|\textbf{v}_{i}\|^2  \| \textbf{v}_c \|^2  \geq  \frac{1}{2}\sum_{i=1}^{N_1} \|\textbf{v}_{i}\|^4  -  \frac{1}{2N_1} (\sum_{i=1}^{N_1} \|\textbf{v}_{i}\|^2)^2  \geq 0,
\end{aligned}
\end{align}
where in the last inequality, we used the fact:
\[ \| \textbf{v}_c \|^2 \leq \frac{1}{N_1} \sum_{i = 1}^{N_1} \| \textbf{v}_{i} \|^2. \]
Then, it follows from \eqref{E-5} and \eqref{E-6} that
\begin{equation} \label{E-8}
{\mathcal I}_{25} \leq  2\delta M_2(\hat V).
\end{equation}
Similarly, one has
\begin{equation} \label{E-9}
{\mathcal I}_{26} \leq  2\delta M_2(\hat W ) .
\end{equation}
In \eqref{E-4}, one can combine all estimates \eqref{E-4-1}, \eqref{E-5}, \eqref{E-8} and \eqref{E-9} to get desired estimates.
\end{proof}

\subsection{Emergence of bi-cluster flocking} \label{sec:5.2} In this subsection, we provide a proof of our third main result on the emergence of bi-cluster flocking for \eqref{A-1} under the Rayleigh friction.

\begin{lemma} \label{L5.2}
\emph{(Macro-dynamics)}
Let $\{ (X, V), (Y, W) \}$ be a solution to \eqref{A-1} with $\delta > 0$. Then, $ \| \textbf{v}_{c} - \textbf{w}_{c} \|$ satisfies
\begin{align*}
\begin{aligned}
\frac{d}{dt} \| \textbf{v}_{c} - \textbf{w}_{c} \| &\geq \left(\frac{2\kappa_d}{N_1 N_2}  \sum_{i = 1}^{N_1} \sum_{j=1}^{N_2}\psi_d(\|\textbf{x}_i -\textbf{y}_j\|) + \delta \right)\| \textbf{v}_c -  \textbf{w}_c \| \\
&\hspace{1cm}- 2\kappa_d \overline{\psi_d} \sqrt{{\hat M}_2} - \delta\sqrt{\max \{N_1,N_2\}} M_2^{\frac{3}{2}}.
\end{aligned}
\end{align*}
\end{lemma}
\begin{proof} Basically, one follows the same argument as in Lemma \ref{L4.2}. First,  taking inner product of $2( \textbf{v}_{c} - \textbf{w}_{c})$ with \eqref{C-5} leads to
\begin{align}
\begin{aligned} \label{E-9-1}
\frac{d}{dt} \| \textbf{v}_{c} - \textbf{w}_{c} \|^2  = & \frac {4\kappa_d}{N_1N_2} \sum_{j=1}^{N_2} \sum_{i=1}^{N_1} \Big( \psi_d(\|\textbf{x}_{i}
- \textbf{y}_{j} \|)  + \delta \Big) \| \textbf{v}_c -  \textbf{w}_c \|^2 \\
&+  \frac {4\kappa_d}{N_1N_2} \sum_{i=1}^{N_1} \sum_{j=1}^{N_2} \psi_d(\|\textbf{y}_{j}
- \textbf{x}_{i} \|)\big(\hat{\textbf{v}}_{i} - \hat{\textbf{w}}_{j} \big) \cdot ( \textbf{v}_{c} - \textbf{w}_{c}) \\
& -\frac{2\delta}{N_1} \sum_{i=1}^{N_1}  \|\textbf{v}_i\|^2 \textbf{v}_i  \cdot ( \textbf{v}_{c} - \textbf{w}_{c}) + \frac{2\delta}{N_2} \sum_{i=1}^{N_2}   \|\textbf{w}_i\|^2 \textbf{w}_i \cdot ( \textbf{v}_{c} - \textbf{w}_{c}) \\
= &\frac {4\kappa_d}{N_1N_2} \sum_{j=1}^{N_2} \sum_{i=1}^{N_1} \Big( \psi_d(\|\textbf{x}_{i}
- \textbf{y}_{j} \|)  + \delta \Big) \| \textbf{v}_c -  \textbf{w}_c \|^2 \\
& + {\mathcal I}_{31} + {\mathcal I}_{32} + {\mathcal I}_{33}.
\end{aligned}
\end{align}

\noindent $\bullet$ (Estimate of ${\mathcal I}_{31}$): By the same argument in Lemma \ref{L4.2}, one has
\begin{equation} \label{E-9-2-0}
|{\mathcal I}_{31}| \leq 2\kappa_d \overline{\psi_d} \sqrt{{\hat M}_2}.
\end{equation}

\noindent $\bullet$ (Estimate of ${\mathcal I}_{32} + {\mathcal I}_{33} $):  For the estimate of the last two terms, note that
\begin{equation} \label{E-9-2}
\left\lVert -\frac{\delta}{N_1} \sum_{i=1}^{N_1} \textbf{v}_i \|\textbf{v}_i\|^2 \right\rVert \leq  \delta \max_i \|\textbf{v}_i\| M_2(V) \leq \delta \sqrt{N_1} M_2(V)^{\frac{3}{2}}.
\end{equation}
Similarly, one has
\begin{equation} \label{E-9-3}
\left\lVert \frac{\delta}{N_2} \sum_{i=1}^{N_2} \textbf{w}_i \|\textbf{w}_i\|^2 \right\rVert\leq \delta \sqrt{N_2} M_2(W)^{\frac{3}{2}}.
\end{equation}
Then, we use \eqref{E-9-2} and \eqref{E-9-3} to get
\begin{align}
\begin{aligned} \label{E-9-4}
& 2 (\textbf{v}_c - \textbf{v}_c) \cdot \left( -\frac{\delta}{N_1} \sum_{i=1}^{N_1} \textbf{v}_i \|\textbf{v}_i\|^2 + \frac{\delta}{N_2} \sum_{i=1}^{N_2} \textbf{w}_i \|\textbf{w}_i \|^2  \right)  \\
& \hspace{1cm} \geq  -  2\left\lVert \textbf{v}_c - \textbf{v}_c \right\rVert
\left\lVert -\frac{\delta}{N_1} \sum_{i=1}^{N_1} \textbf{v}_i \|\textbf{v}_i\|^2 + \frac{\delta}{N_2} \sum_{i=1}^{N_2} \textbf{w}_i \|\textbf{w}_i \|^2 \right\rVert   \\
& \hspace{1cm}  \geq  -  2\left\lVert \textbf{v}_c - \textbf{w}_c \right\rVert \delta \sqrt{\max \{N_1,N_2\}} M_2(t)^{\frac{3}{2}}.
\end{aligned}
\end{align}
In \eqref{E-9-1}, one can combine all estimates \eqref{E-9-2-0} and \eqref{E-9-4} to obtain the desired estimate.
\end{proof}
Now, we are ready to present our third main result on the emergence of bi-cluster flocking.
\begin{theorem} \label{T5.1}

Suppose that $\delta>0$, the initial center velocity are separated enough and the intra-interaction outweighs the inter-interaction, namely
\[ \|\textbf{v}_{c}(0) - \textbf{w}_{c}(0) \|  \gg 1, \quad  \sup_{0 \leq t < \infty} \Big( \delta + \kappa_d \overline{\psi_d}(t) - \kappa_s \underline{\psi_s}(t) \Big) \leq -\eta_1. \]
Then, we have a bi-cluster flocking in the sense of Definition \ref{D1.1}.
\end{theorem}
\begin{proof}  Below, we present a bi-cluster flocking estimate in several steps. \newline

\noindent $\bullet$ Step A (Uniform upper bound of $\| \textbf{v}_c - \textbf{w}_c \|$):
By Cauchy's inequality and Corollary \ref{C2.1},
\begin{align}
\begin{aligned} \label{E-10}
\| \textbf{v}_{c} - \textbf{w}_{c} \| & = \Big \| \frac{1}{N_1} \sum_{i=1}^{N_1} \textbf{v}_i -\frac{1}{N_2}\sum_{i=1}^{N_2} \textbf{w}_i \Big \| \leq  \frac{1}{N_1} \sum_{i=1}^{N_1} \| \textbf{v}_i  \| + \frac{1}{N_2}\sum_{i=1}^{N_2} \| \textbf{w}_i \| \\
& \leq 2 \sqrt{ \frac{1}{N_1} \sum_{i=1}^{N_1} \| \textbf{v}_i \|^2+ \frac{1}{N_1} \sum_{i=1}^{N_2} \|\textbf{w}_i} \|^2 = 2\sqrt{M_2(t)} \leq 2\sqrt{M_2^\infty}.
\end{aligned}
\end{align}
\vspace{0.2cm}
\noindent $\bullet$ Step B (Uniform boundedness of $\hat M_2$): We combine Lemma \ref{L5.1} and \eqref{E-10} to obtain
\begin{align} \label{E-11-1}
 \frac{d \sqrt{\hat M_2(t)}}{dt} &\leq -\eta_1 \sqrt{\hat M_2(t)} +  \kappa_d \overline{\psi_d}(t)  \| \textbf{v}_{c} - \textbf{w}_{c} \| \leq -\eta_1 \sqrt{\hat M_2(t)} +  2 \kappa_d \overline{\psi_d}(t) \sqrt{M_2^\infty}
\end{align}
Note that here $\overline{\psi_d}(t)$ still depends on $t$. In this step we simply use a rough bound $\overline{\psi_d}(t) \leq \psi_d^\infty $, and use Gronwall's lemma to obtain
\begin{align} \label{E-11}
\begin{aligned}
\sqrt{\hat M_2 (t)} & \leq \sqrt{\hat M_2 (0)} e^{-\eta_1 t} + \frac{2
\kappa_d \psi_d^\infty \sqrt{M_2^\infty}}{\eta_1} (1-e^{-\eta_1 t}) \\
& \leq \max\{ \sqrt{\hat M_2 (0)}, \frac{2\kappa_d \psi_d^\infty\sqrt{M_2^\infty}}{\eta_1} \} : = C_6.
\end{aligned}
\end{align}
\vspace{0.2cm}

\noindent $\bullet$ Step C (Separation of the center velocities): By Lemma \ref{L5.2}, Corollary \ref{C2.1} and \eqref{E-11},
\begin{align*}
\begin{aligned}
~\frac{d}{dt} \| \textbf{v}_{c} - \textbf{w}_{c} \| \geq & \left(\frac{2\kappa_d}{N_1 N_2}  \sum_{i = 1}^{N_1} \sum_{j=1}^{N_2}\psi_d(\|x_i -y_j\|) + \delta \right)\| \textbf{v}_c -  \textbf{w}_c \| \\
&- 2\kappa_d \overline{\psi_d} C_6- \delta\sqrt{\max \{N_1,N_2\}} (M_2^\infty)^{\frac{3}{2}}.
\end{aligned}
\end{align*}
By Gronwall's lemma, we have a rough bound:
\[ \| \textbf{v}_{c}(0) - \textbf{w}_{c}(0) \| \geq \frac{2\kappa_d \psi_d^{\infty} C_6+\delta\sqrt{\max \{N_1,N_2\}} (M_2^\infty)^{\frac{3}{2}}}{\delta} : = \frac{C_7}{\delta}, \]
The center velocities stay separated for later time
\begin{equation*} 
 \| \textbf{v}_{c}(t) - \textbf{w}_{c}(t) \|  \geq \frac{C_7}{\delta} + \left(\| \textbf{v}_{c}(0) - \textbf{w}_{c}(0) \| -\frac{C_7}{\delta} \right)e^{\delta t} \geq \frac{C_7}{\delta}.
\end{equation*}
\vspace{0.2cm}

\noindent $\bullet$ Step D (Spatial separation of the two sub-ensembles):  First, we consider for any $i = 1, \cdots, N_1$ and $j = 1, \cdots, N_2$,
\begin{align*}
\begin{aligned}
\|\textbf{v}_i(t) -\textbf{w}_j(t)\| \geq &\|\textbf{v}_c(t) - \textbf{w}_c(t)\| - \|\hat{\textbf{v}}_i(t) - \hat{\textbf{w}}_j(t)\| \\
 \geq &\|\textbf{v}_c(t) - \textbf{w}_c(t)\| - \sqrt{2 \max\{N_1,N_2\} \hat M_2} \\
\geq &\frac{C_7}{\delta} + \Big(\| \textbf{v}_{c}(0) - \textbf{w}_{c}(0) \| - \frac{C_7}{\delta} \Big)e^{\delta t} \\
& - \sqrt{2 \max \{N_1,N_2\}} \Big(\sqrt{\hat M_2 (0)} e^{-\eta_1 t}  + \frac{\sqrt{M_2^\infty}}{\eta_1} (1-e^{-\eta_1 t}) \Big).
\end{aligned}
\end{align*}
Then, there exists some time $T^\ast$ and constant $C_8$ such that for any time $t \geq T^\ast$,
\[ \|\textbf{v}_i(t) -\textbf{w}_j(t)\| \geq C_8 >0. \]
As a direct consequence, for $t \geq T^\ast$ the minimum distance between the two groups has at least linear growth with respect to time, namely,
\begin{align*}
\min_{i \in {\mathcal G}_1, j \in {\mathcal G}_2}\|\textbf{x}_i(t) - \textbf{y}_j(t) \| \geq C_8 t + \gamma_0.
\end{align*}
Hence, we have
\begin{align} \label{E-13}
\overline{\psi}_d(t) \leq \psi_d \left(\min_{i \in {\mathcal G}_1, j \in {\mathcal G}_2}\|\textbf{x}_i(t) - \textbf{y}_j(t) \| \right) \leq \psi_d \left( C_8 t + \gamma_0 \right) \to 0,  \quad \mbox{as $t \to \infty$}.
\end{align}
Note that here the sign of $\gamma_0$ does not matter. \newline

\noindent $\bullet$ Step E (Emergence of the bi-cluster flocking): In this step, we use \eqref{E-13} to improve the estimate of the velocity fluctuations \eqref{E-11-1}. To be specific, one has
\begin{align*}
 \frac{d \sqrt{\hat M_2(t)}}{dt} & \leq -\eta_1 \sqrt{\hat M_2(t)} +  \kappa_d \sqrt{M_2^\infty} \psi_d \left( C_8 t + \gamma_0 \right).
\end{align*}
The Gronwall-type lemma in Lemma \ref{LA.2} yields
\begin{align*} \label{E-14}
\begin{aligned}
\sqrt{\hat M_2(t)} \leq & \frac{1}{\eta_1}  \max_{s \in [t/2, t]} \kappa_d \sqrt{M_2^\infty}|\psi_d \Big( C_8 s + \gamma_0 \Big)| \\
&+\sqrt{\hat M_2(0)} e^{-\eta_1 t}+\frac{ \kappa_d \sqrt{M_2^\infty} \psi_d^\infty}{\eta_1}e^{-\frac{\eta_1 t}{2}} \quad t \geq T^\ast.
\end{aligned}
\end{align*}
This completes the bi-cluster flocking estimate.
\end{proof}
\begin{remark}
Note that the condition $\delta > 0$ is crucially used to show the uniform boundedness of $M_2$ (see Corollary \ref{C2.1}). Otherwise, $M_2$ may not be bounded uniformly in time as can be seen from the explicit example in Section \ref{sec:3.1} and one needs to impose some restrictions on $\psi_s$ and $\psi_d$ as in Theorem \ref{T4.1}.
\end{remark}
\section{Numerical experiments} \label{sec:6}
\setcounter{equation}{0}
In this section, we perform several numerical experiments to confirm the analytical results presented in previous sections. For numerical implementation, the fourth order Runge-Kutta scheme is used.

\subsection{The growth of the kinetic energy}  \label{sec:6.1}
As mentioned in Remark \ref{T4.1}, the Rayleigh friction term is crucial for the uniform bound of the kinetic energy $M_2$. In the case of $\delta = 0$, $M_2$  could grow exponentially initially, and keep growing even after flocking phenomenon occurs.

\begin{example}[Case I: $\delta=0$] \label{delta0_sd_same}
Consider the case $\delta=0$ with mixed initial configurations in both position and velocity. Here the initial positions of the two sub-ensembles are both generated randomly in the unit square $[0,1] \times [0,1]$ while the initial velocities are also randomly generated and then normalized to mean 0, as  plotted in Fig. \ref{initial_config_delta_0}. 
\begin{figure}[htb]
\centering
\subfloat[Spatial configuration with the arrows presenting the velocity fields]{\includegraphics[width=0.48\textwidth]{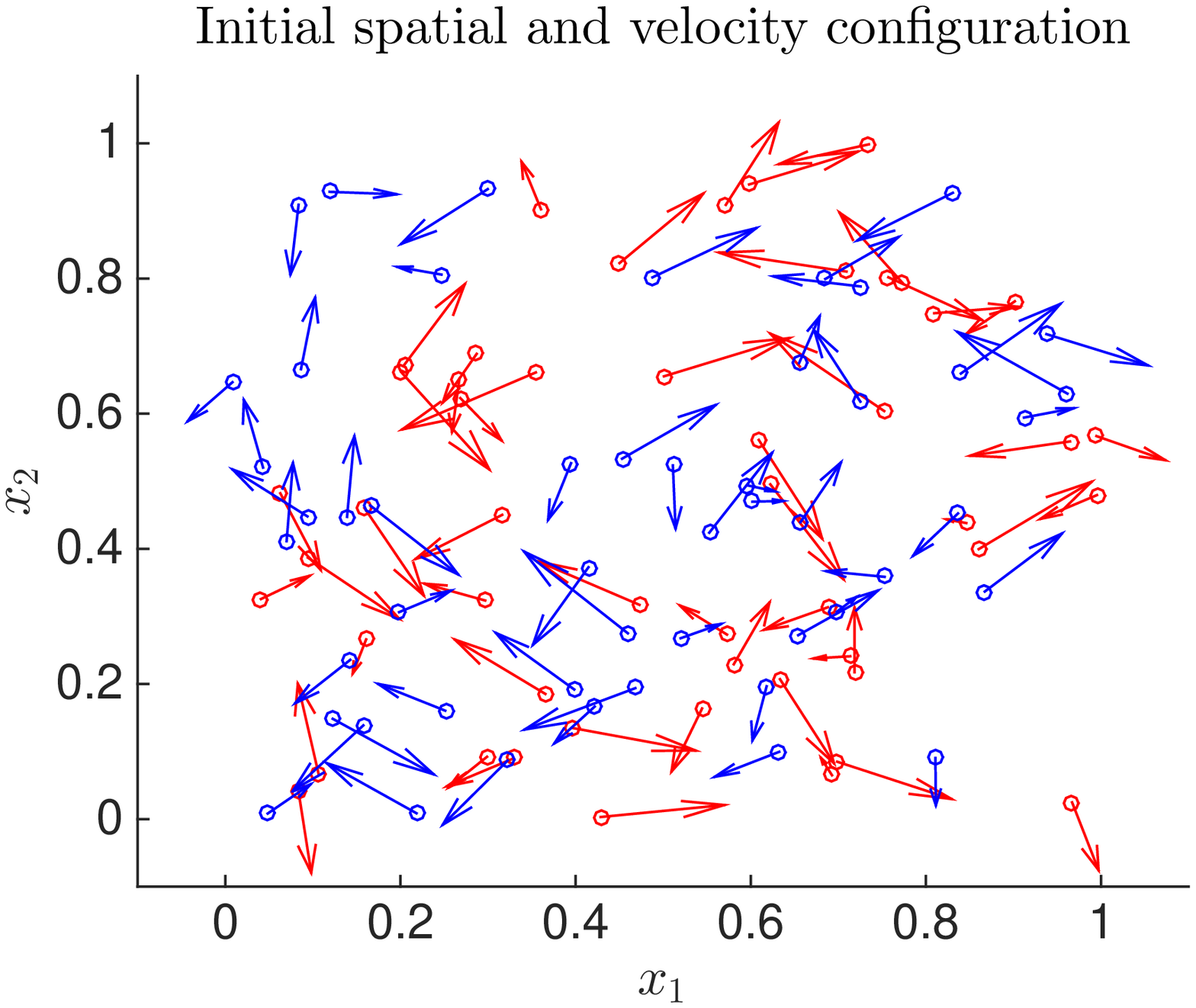}}
\subfloat[Velocity configuration in 2D velocity plane ]{\includegraphics[width=0.48\textwidth]{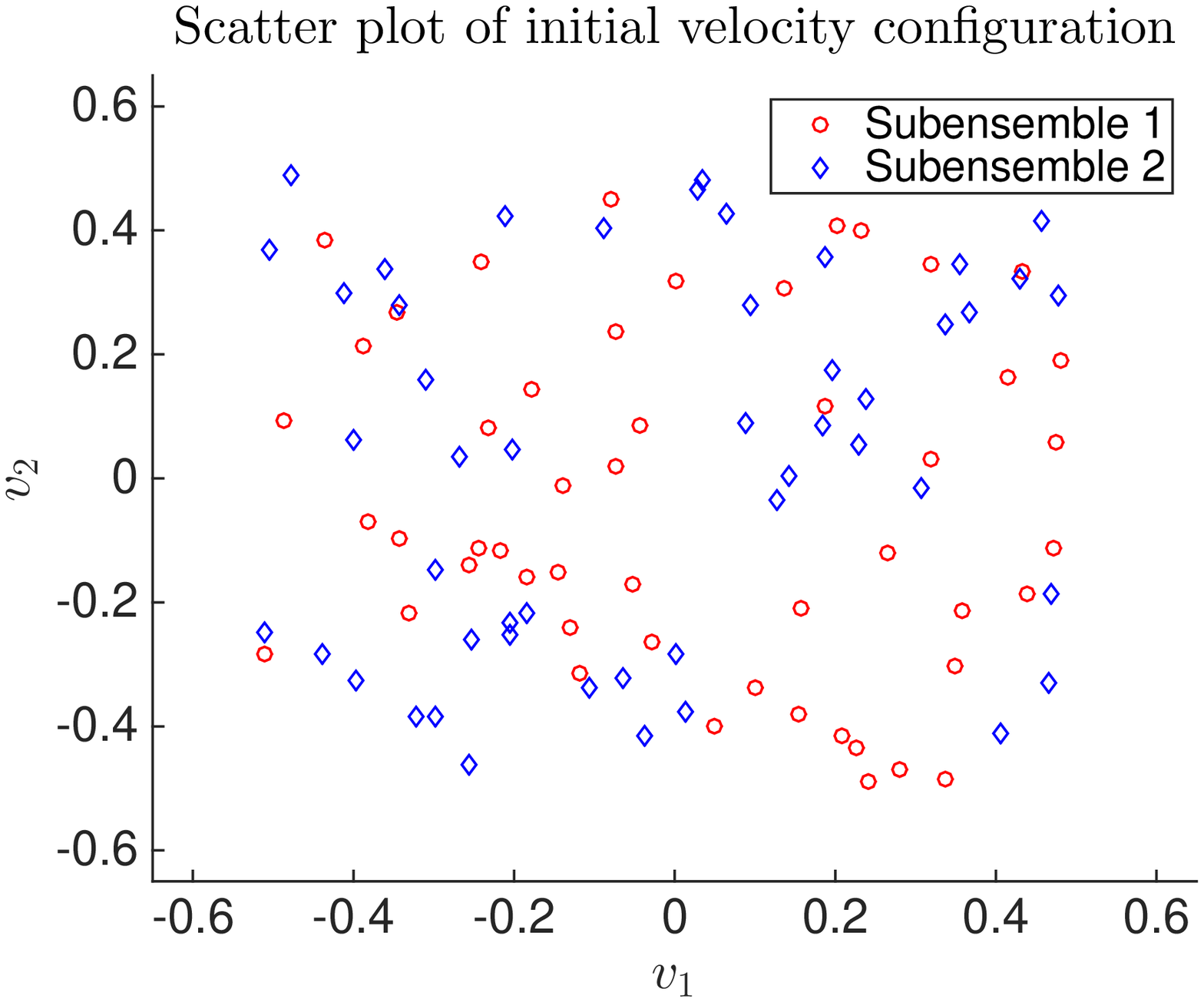}}
\caption{Example \ref{delta0_sd_same}: Initial spatial and velocity mixed configurations }
\label{initial_config_delta_0}
\end{figure}
The parameters of the two groups are chosen as $N_1 = N_2 = 50$, $\kappa_s = \kappa_d = 10$, and the communication functions are
$$\psi_s(r) = \psi_d(r) = \frac{1}{(1+r^2)^{\beta}},$$
with $\beta = 0.4$. Note that this is the commonly-used communication function in the computation and sometimes analysis of the flocking dynamics. In this example, the kinetic energy $M_2$ first grows exponentially, and it keeps growing after spatial separation and even till velocity fluctuations are about 0. To make it clearer, the kinetic energy $M_2$ and velocity fluctuation $\hat{M}_2$ are plotted side by side in Fig.\ref{delta0_kinetic}, and the initial growth of kinetic energy is zoomed in.
Note that this is an interesting phenomenon that is different from many traditional mono-flocking models where the kinetic energy decreases and hence can be trivially bounded by its initial data. The growth of kinetic energy makes our model interesting and also more difficult from analysis point of view.

\begin{figure}[htbp]
\centering
\subfloat[Kinetic energy $M_2$ till flocking occurs]{\includegraphics[width=0.48\textwidth]{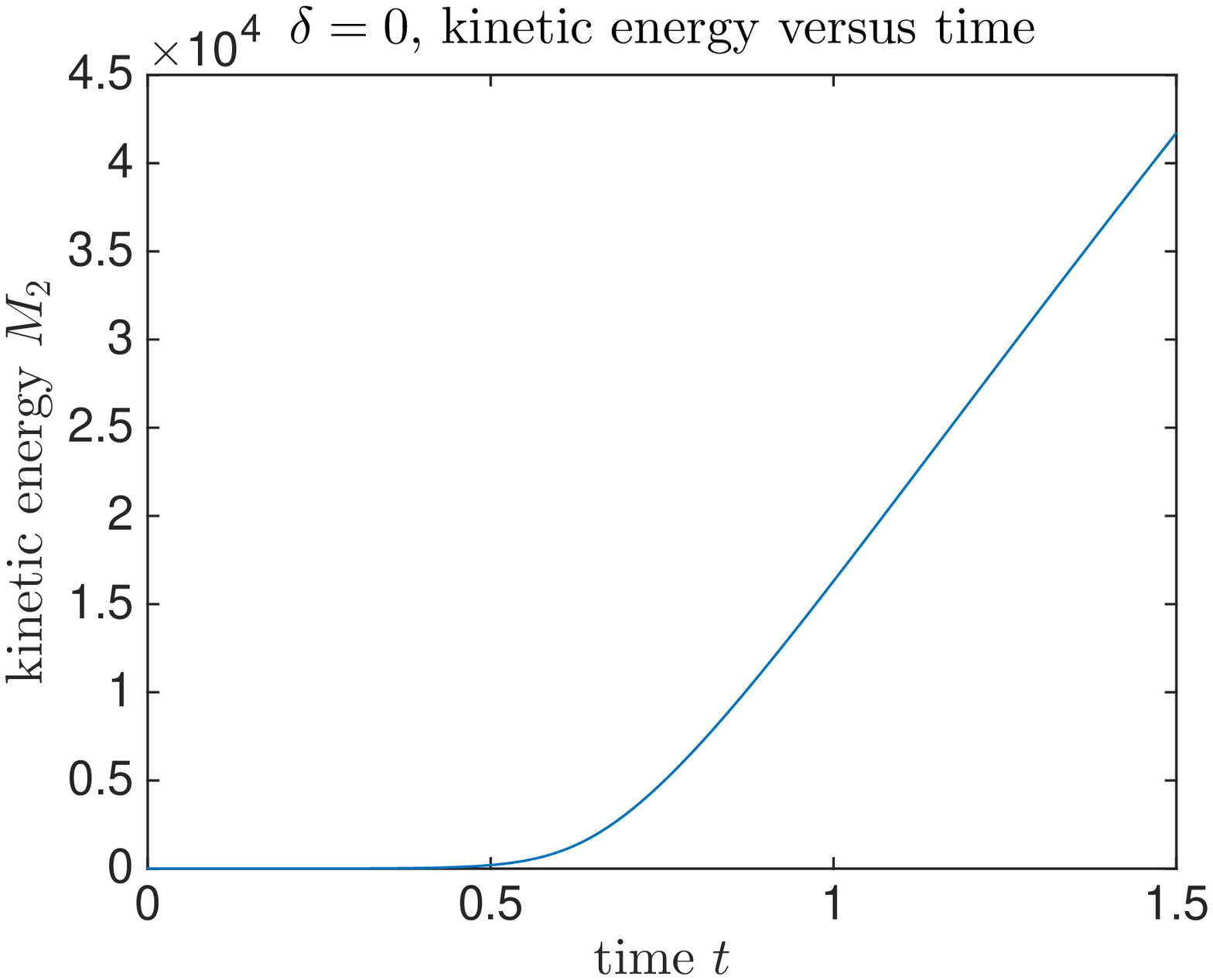}}
\subfloat[Velocity fluctuations ]{\includegraphics[width=0.48\textwidth]{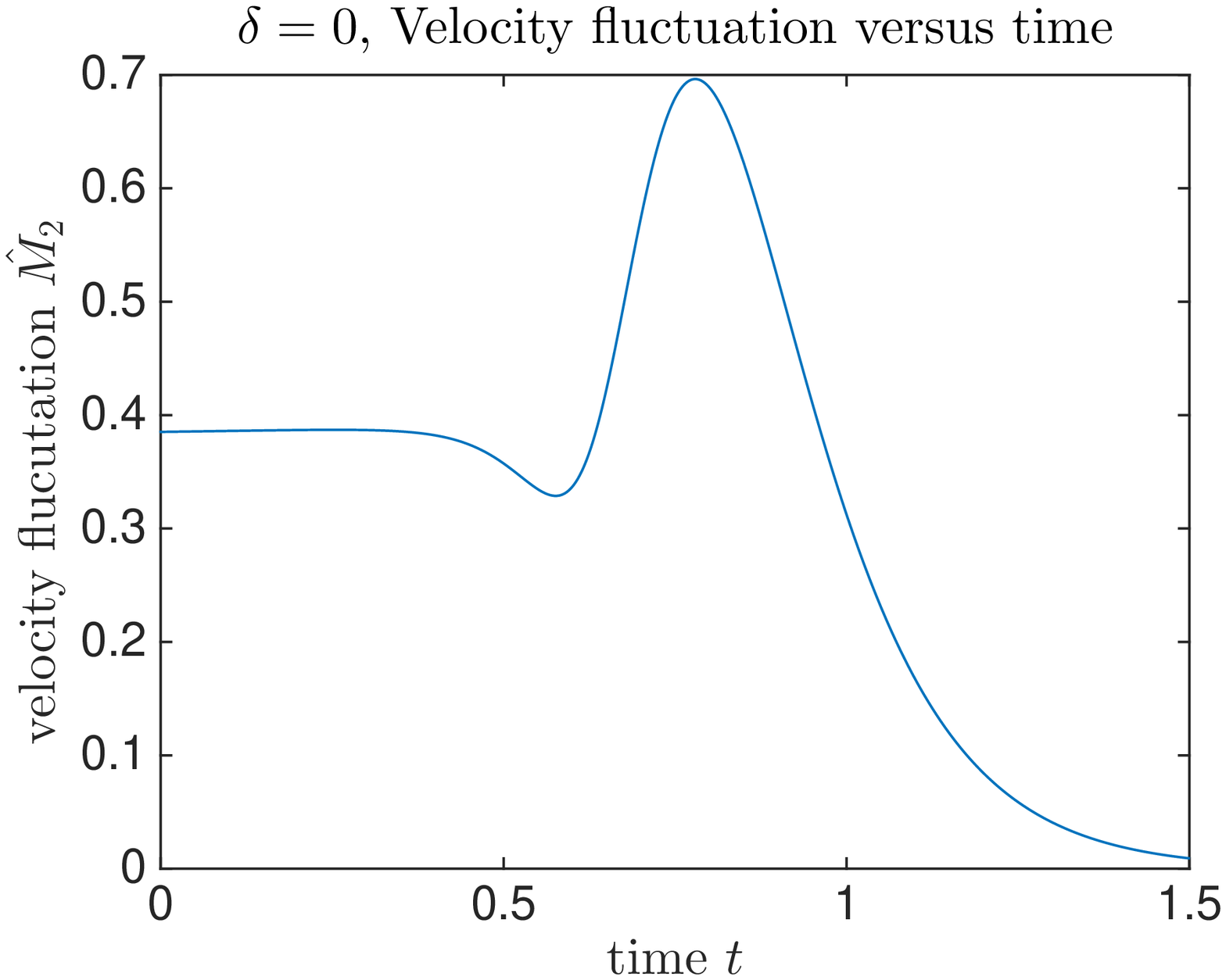}} \\
\subfloat[Zoomed in of the initial growth  of the kinetic energy in the usual scale and semilogy scale]{\includegraphics[width=.82\textwidth]{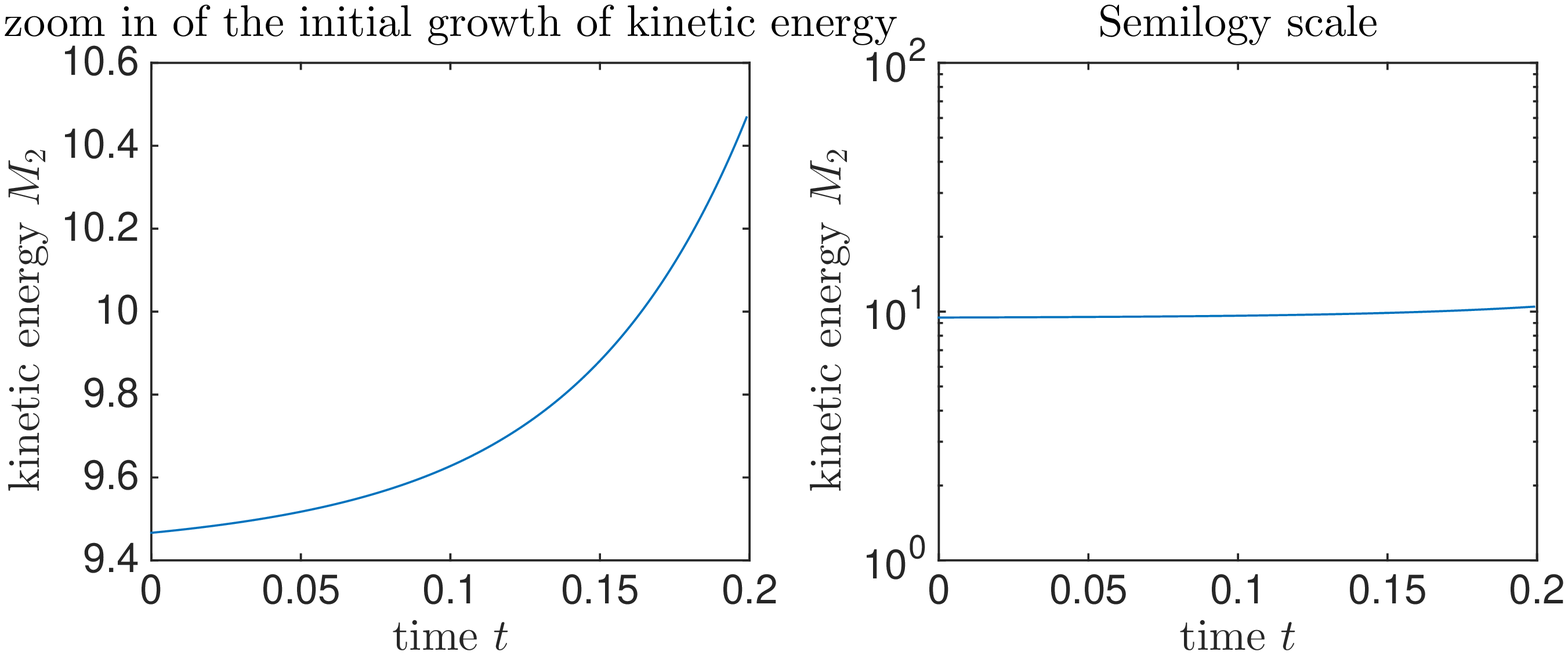}}
\caption{Example \ref{delta0_sd_same}: The kinetic energy $M_2$ grows exponentially at the beginning, as is made clear in the plot of semilogy scale, and then keeps growing even till the velocity fluctuations are about 0.}
\label{delta0_kinetic}
\end{figure}

\end{example}

\begin{example}[Case II: $\delta>0$] \label{delta1_sd_same}
Now consider the same parameters, as well as the spatial and velocity mixed initial configurations in Example \ref{delta0_sd_same} but turn on the Rayleigh friction, namely $\delta>0$. It can be seen in Fig. \ref{kinetic_rayleigh} that the existence of the Rayleigh friction drastically changes the profile of the kinetic energy. Fig. \ref{kinetic_rayleigh} plots the kinetic energy $M_2$ with $\delta = 1$ and $\delta = 0.1$, respectively. In both cases although $M_2$  is not monotone -- as in many traditional mono-flocking models, it decays eventually and hence still has a uniform bound. For larger $\delta$, the kinetic energy is bounded by initial data while for smaller $\delta$, the uniform bound is different, which agrees with the estimate in Corollary \ref{C2.1}.
\begin{figure}[htbp]
\centering
\subfloat[$\delta = 1$]{\includegraphics[width=0.5\textwidth]{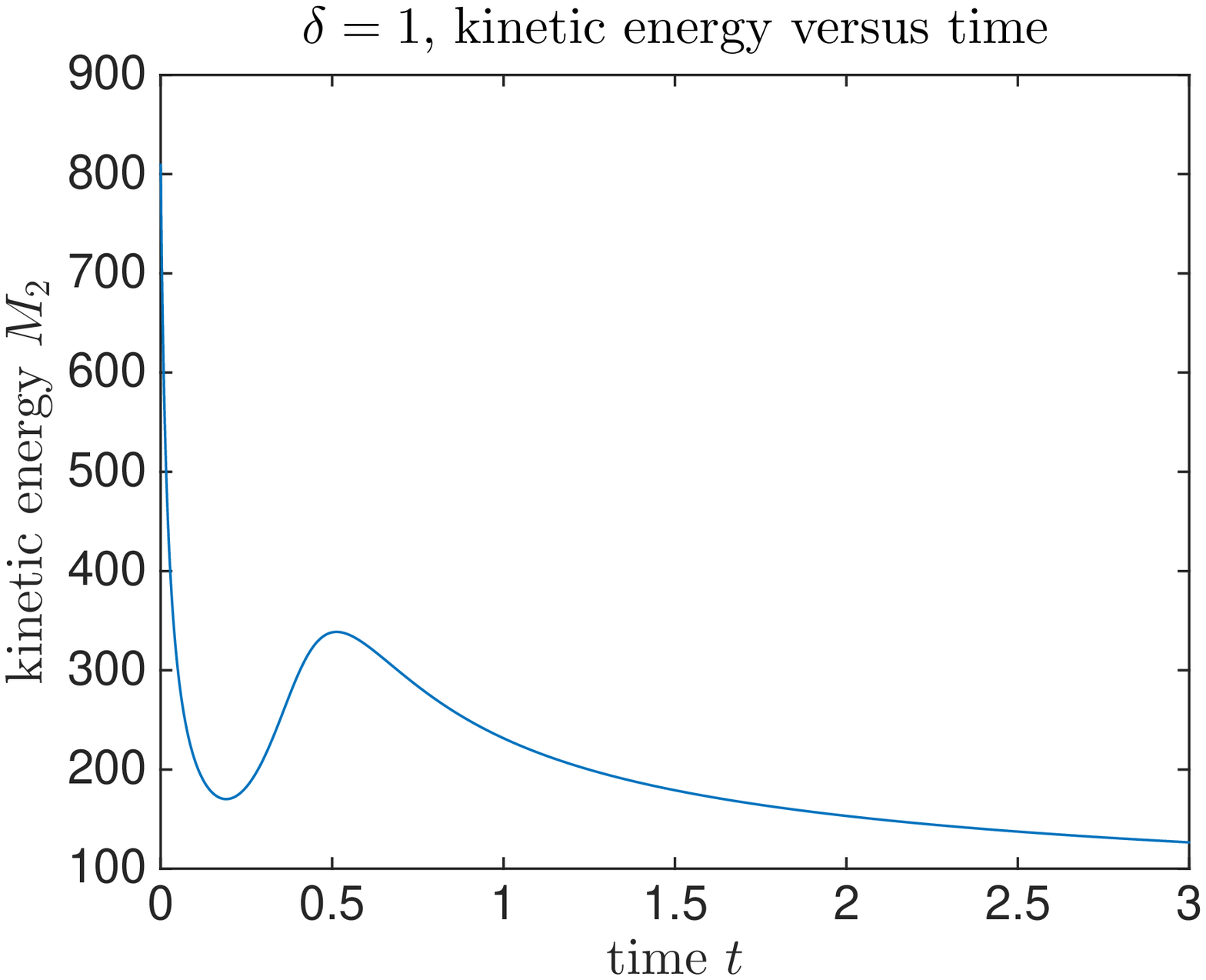}}
\subfloat[$\delta = 0.1 $ ]{\includegraphics[width=0.5\textwidth]{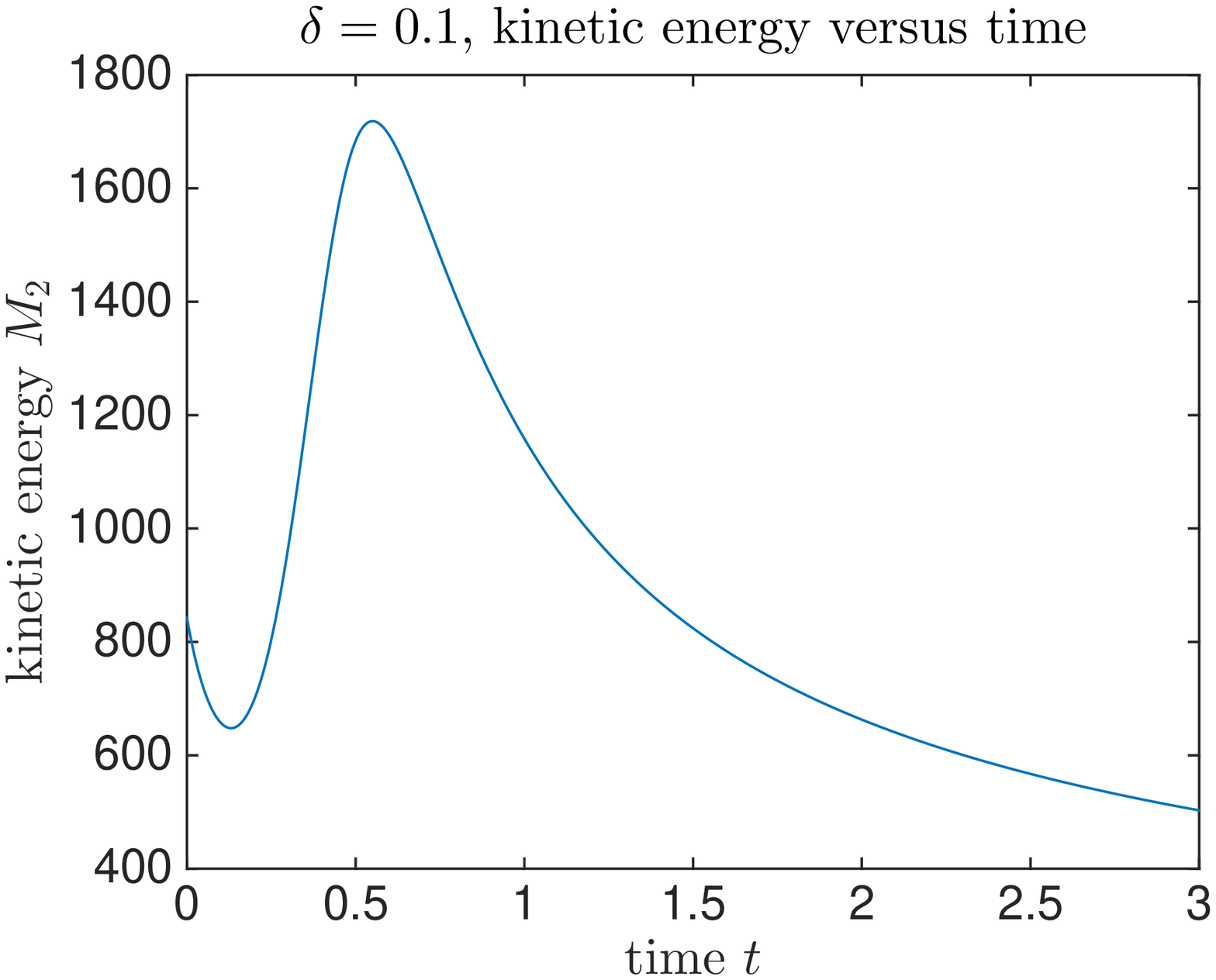}} \\
\caption{Example \ref{delta1_sd_same}: The kinetic energy $M_2$  is not monotone for $\delta > 0$, but are still uniformly bounded.}
\label{kinetic_rayleigh}
\end{figure}
\end{example}

\subsection{The emergence of bi-clustering}
\begin{example}[Different Stages for $\delta = 0$] \label{ex_stage1}
In this example, we shall show different stages of the emergence of bi-cluster flocking. Consider the spatially and velocity mixed initial configuration in Fig. \ref{stage_initial}, the bi-clustering phenomenon occurs in the following stages: (a) Stage 1: velocity separation of two sub-ensembles; (b) Stage 2: spatial separation of two sub-ensembles; (c) Stage 3: emergence of bi-cluster flocking.

\begin{figure}[htbp]
\centering
\subfloat[Stage 0: Initial configuration.]{\includegraphics[width=0.5\textwidth]{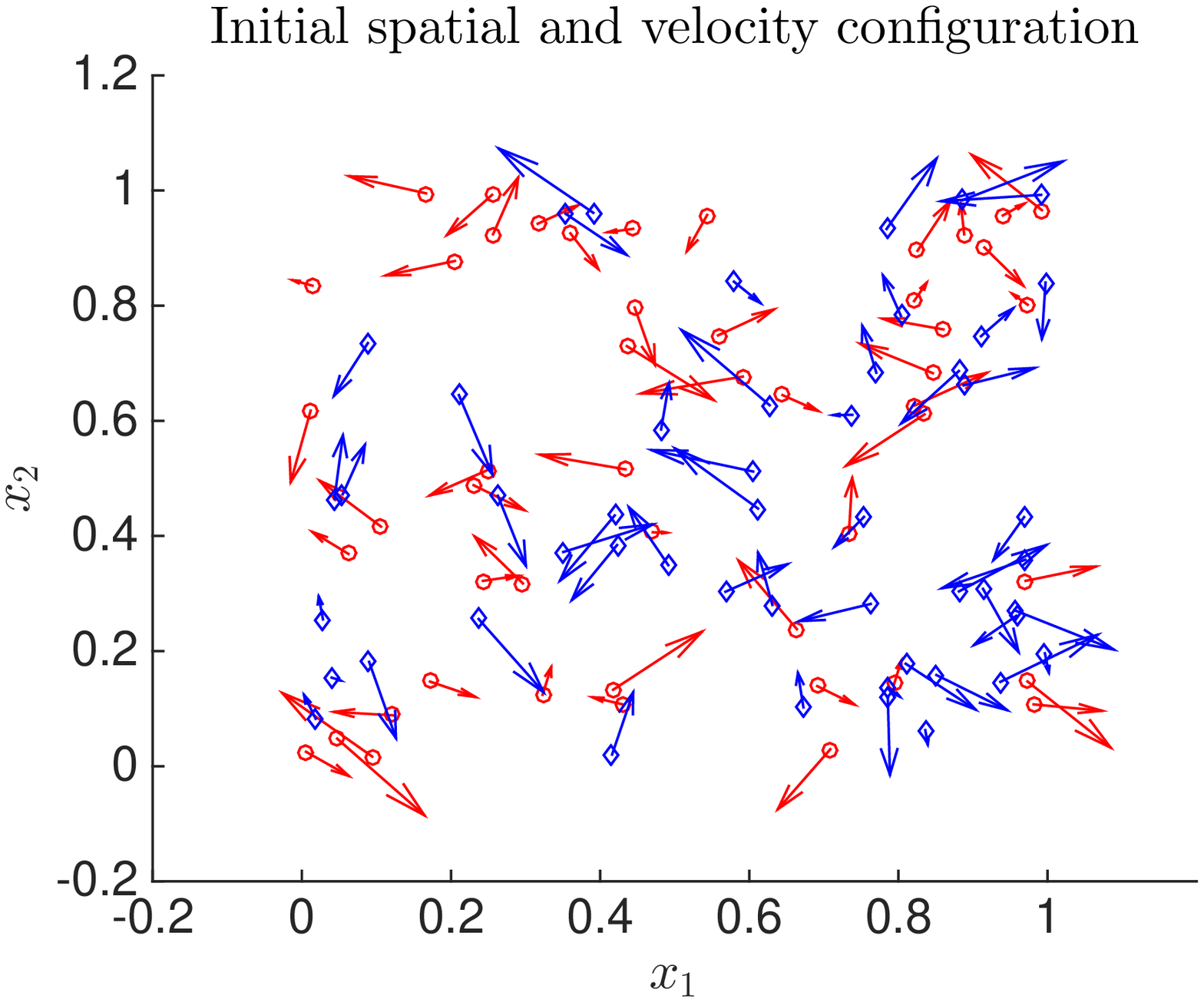}}
\subfloat[Initial velocities. ]{\includegraphics[width=0.5\textwidth]{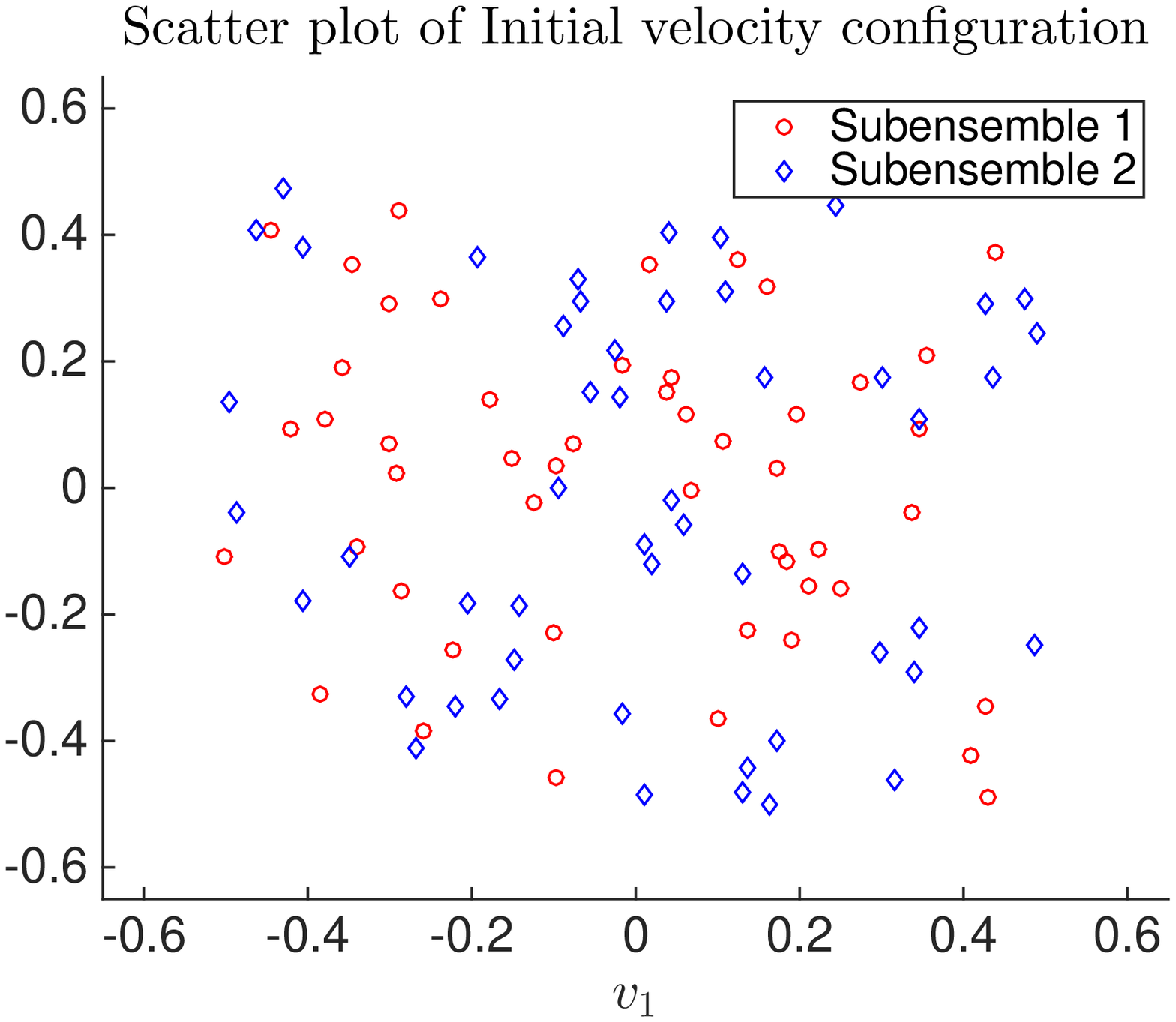}}
\caption{Initial configurations of Example \ref{ex_stage1} and \ref{ex_stage2}.}
\label{stage_initial}
\end{figure}
\begin{figure}[htbp]
\centering
\subfloat[Stage 1: velocity separation.]{\includegraphics[width=0.5\textwidth]{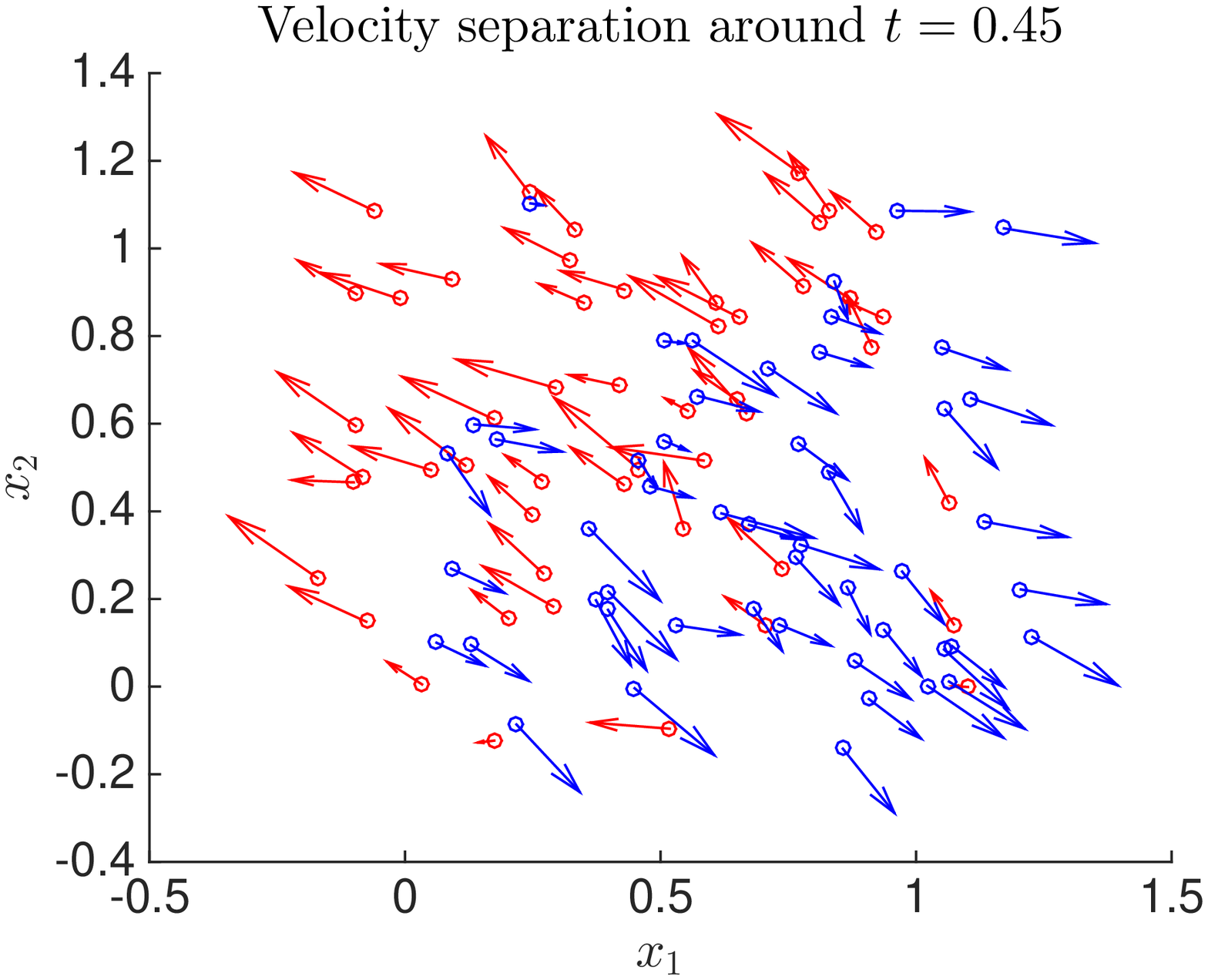}}
\subfloat[Plot of the velocity difference between the two subensembles, where velocity separation happens around $t = 0.45$.]{\includegraphics[width=0.5\textwidth]{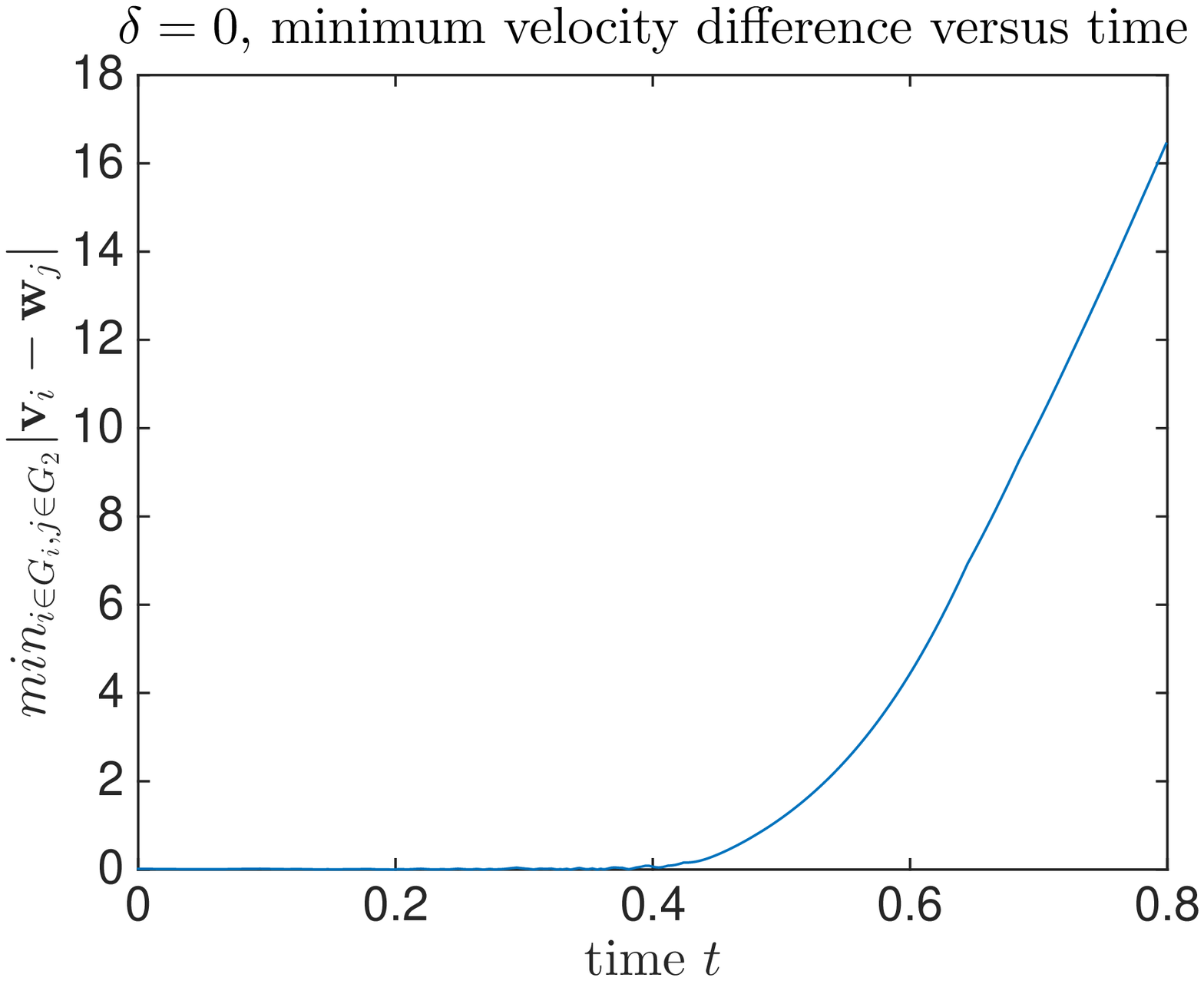}} \\
\subfloat[Stage 2: spatial separation]{\includegraphics[width=0.5\textwidth]{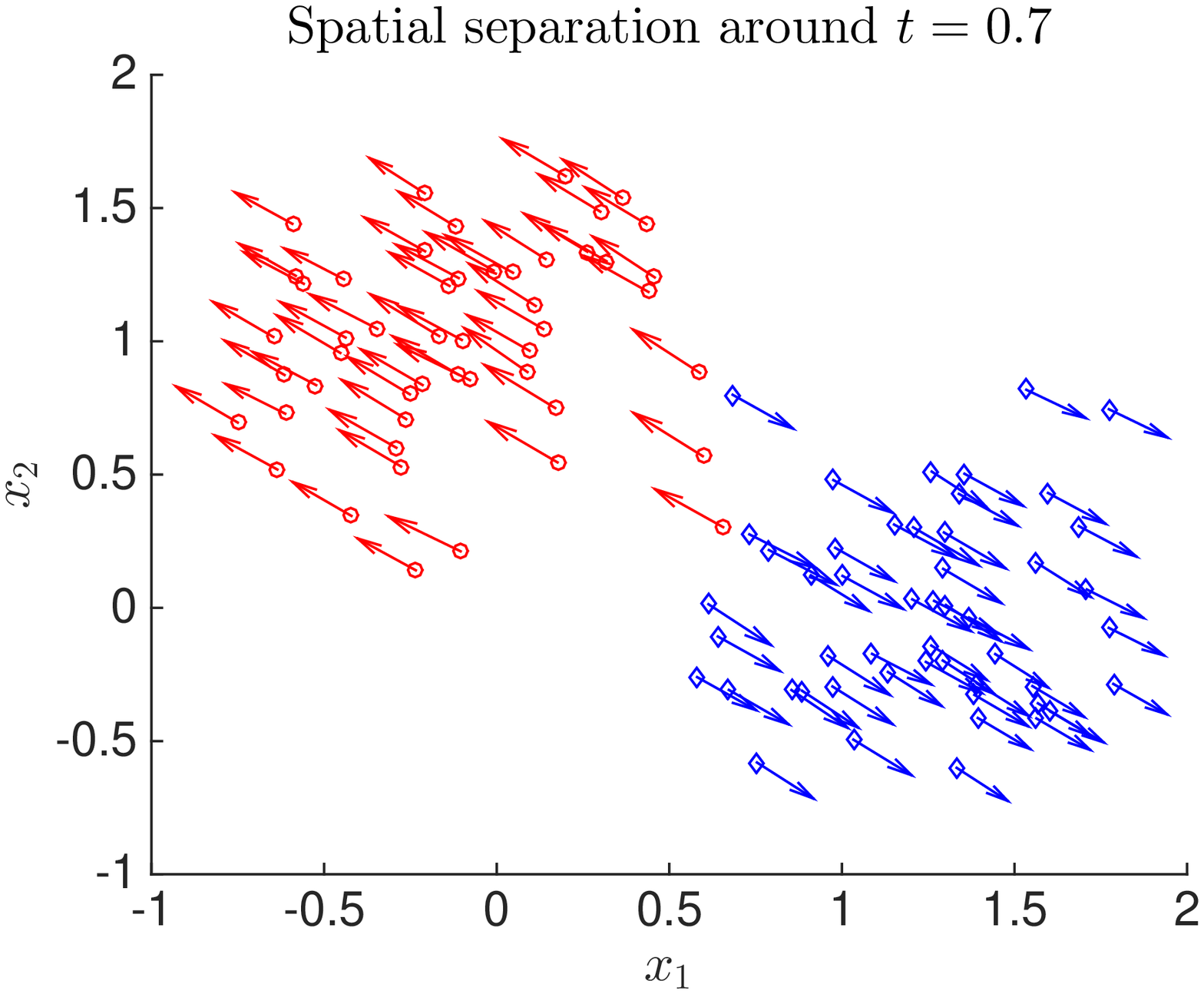}}
\subfloat[Plot of the spatial difference between the two subensembles, where spatial separation happens around $t = 0.7$.]{\includegraphics[width=0.5\textwidth]{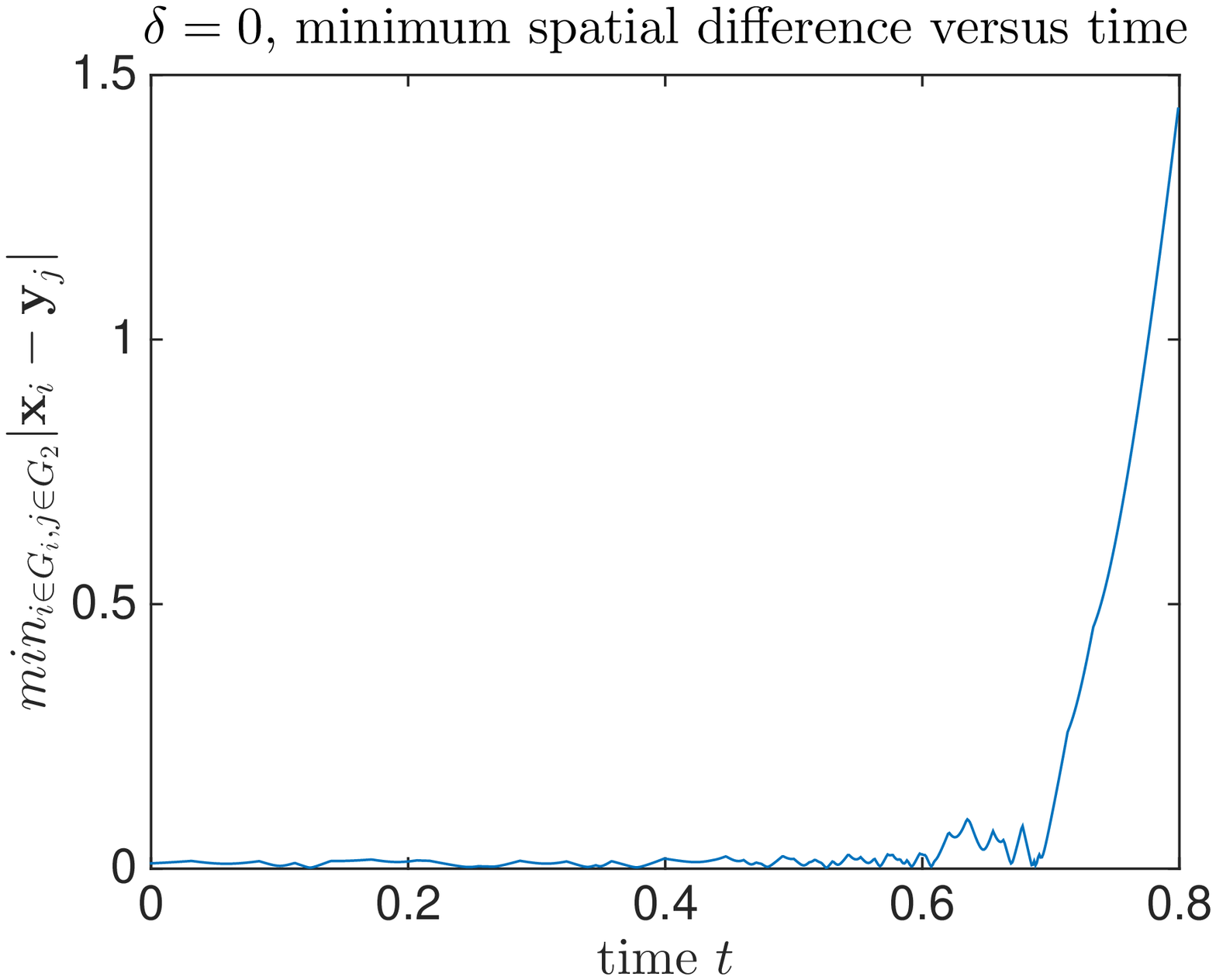}} \\
\subfloat[Stage 3: Emergence of bi-cluster flocking]{\includegraphics[width=0.5\textwidth]{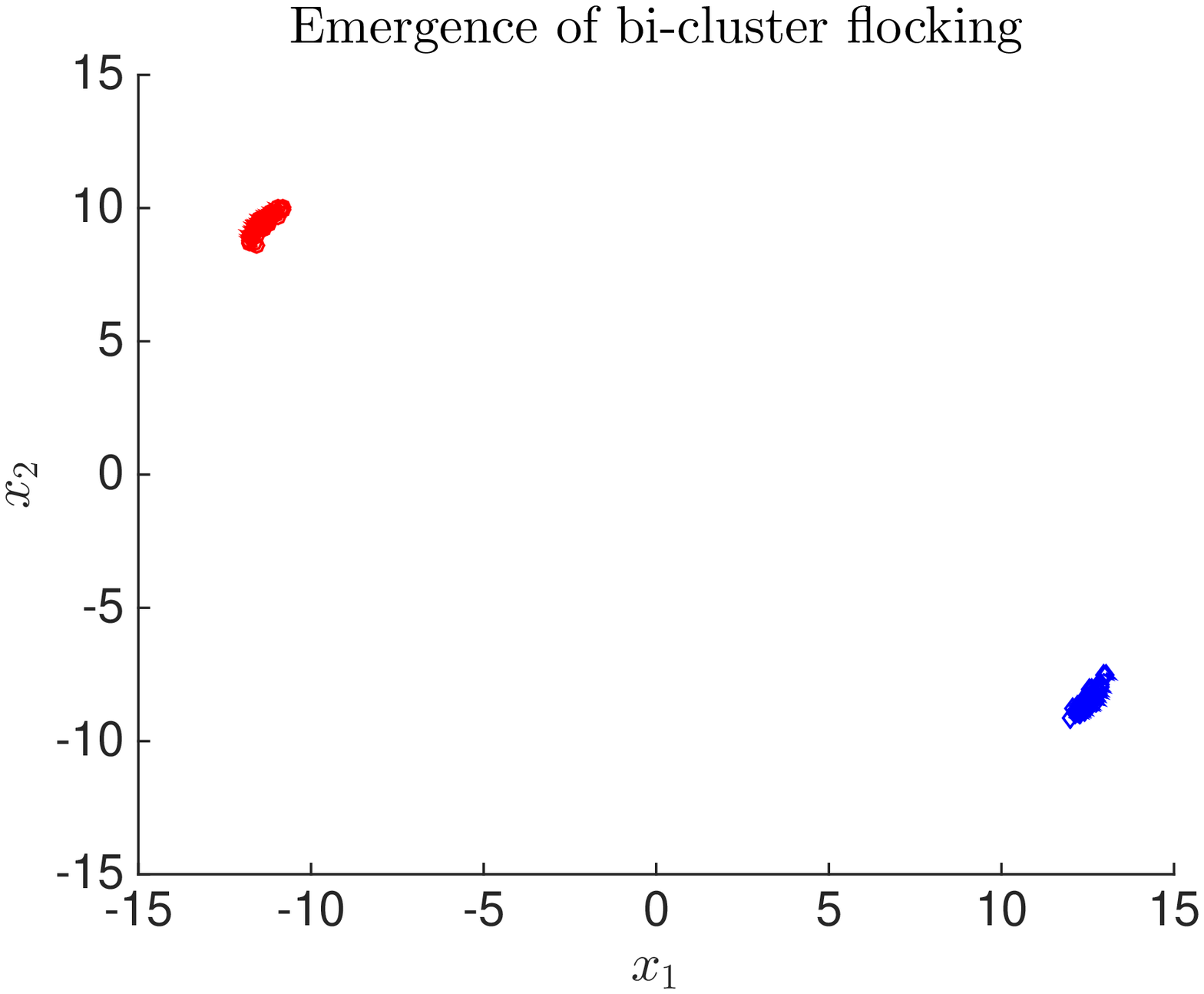}}
\subfloat[Plot of velocity fluctuations]{\includegraphics[width=0.5\textwidth]{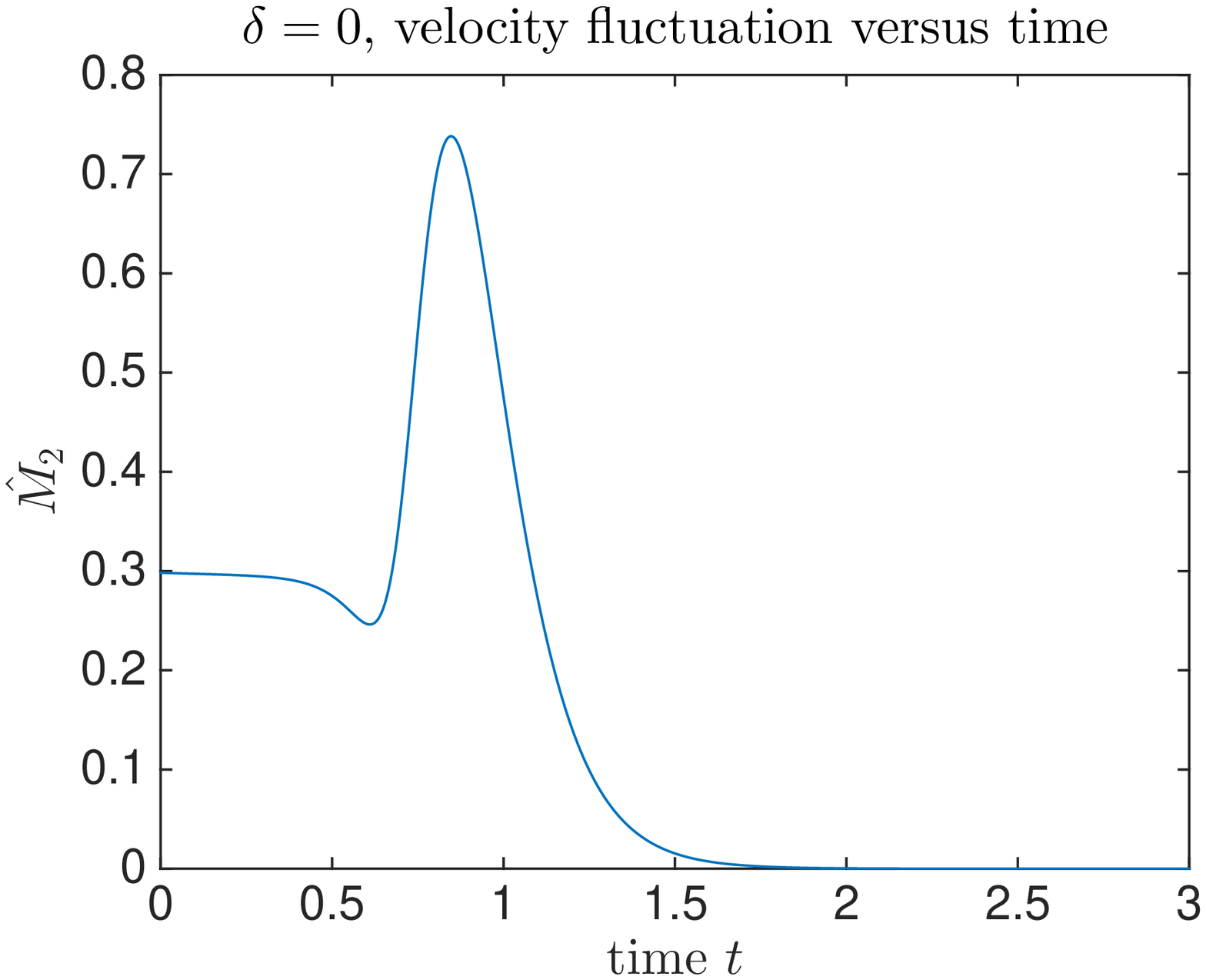}}
\caption{Example \ref{delta1_sd_same}: $\delta = 0$. From mixed initial configuration, the emergence of flocking occurs in three stages as plotted.}
\label{delta0_stage}
\end{figure}

\end{example}

\begin{example}[Different Stages for $\delta > 0$] \label{ex_stage2}
In this example, the Rayleigh friction is turned on with $\delta = 1$. We test the influence incurred by the Rayleigh friction. Consider the spatially and velocity mixed initial configuration in Fig. \ref{stage_initial}. Similar as previous example, the bi-clustering phenomenon occurs in the following stages: (a) Stage 1: velocity separation of two sub-ensembles; (b) Stage 2: spatial separation of two sub-ensembles; (c) Stage 3: emergence of bi-cluster flocking. It can be seen from Fig. \ref{delta1_stage} that the existence of Rayleigh friction does not change qualitatively the three stages, but make the initial layer of the velocity fluctuation (the time interval before it decays) narrower.

\begin{figure}[htbp]
\centering
\subfloat[Stage 1: velocity separation.]{\includegraphics[width=0.5\textwidth]{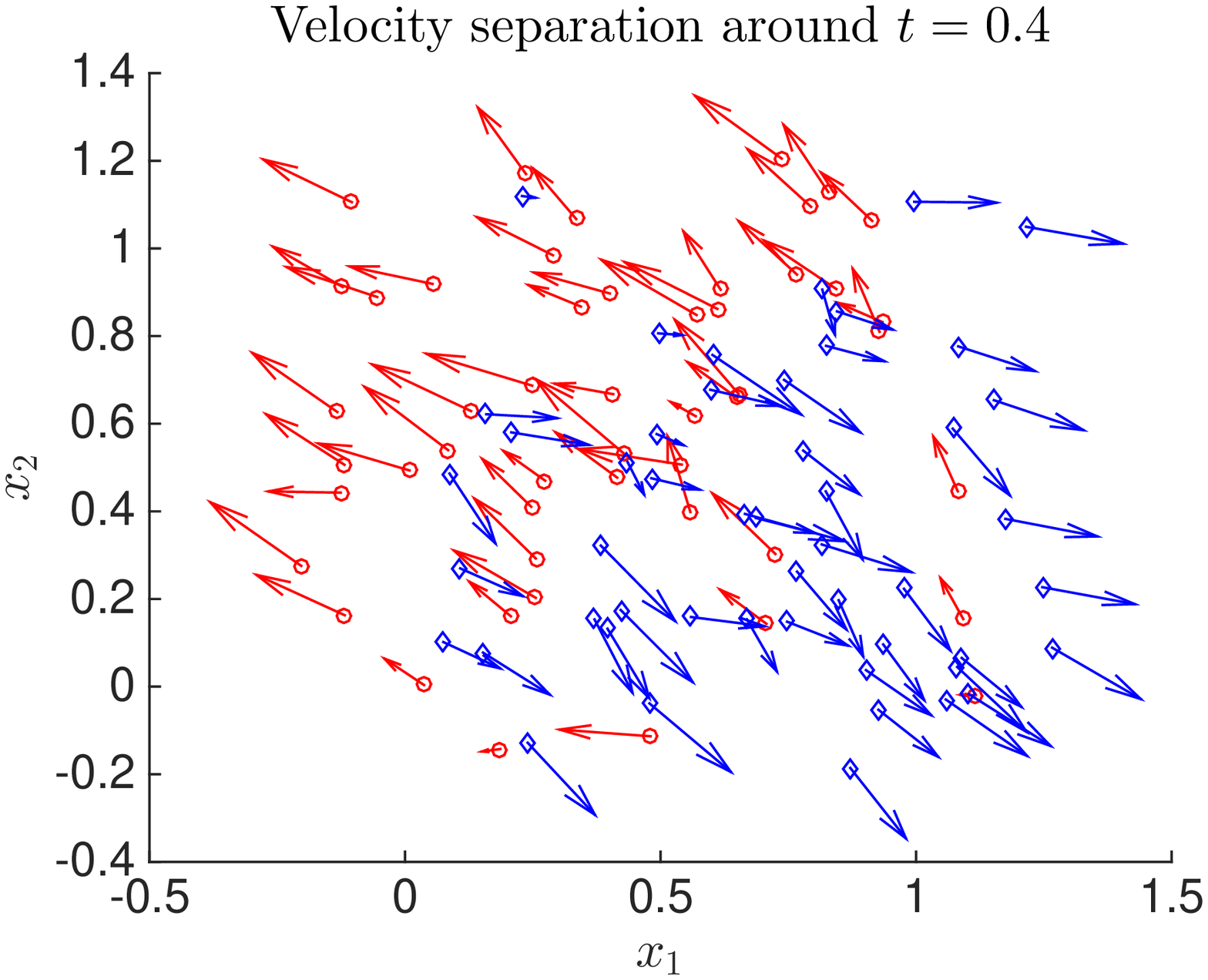}}
\subfloat[Plot of the velocity difference between the two subensembles, where velocity separation happens around $t = 0.45$.]{\includegraphics[width=0.5\textwidth]{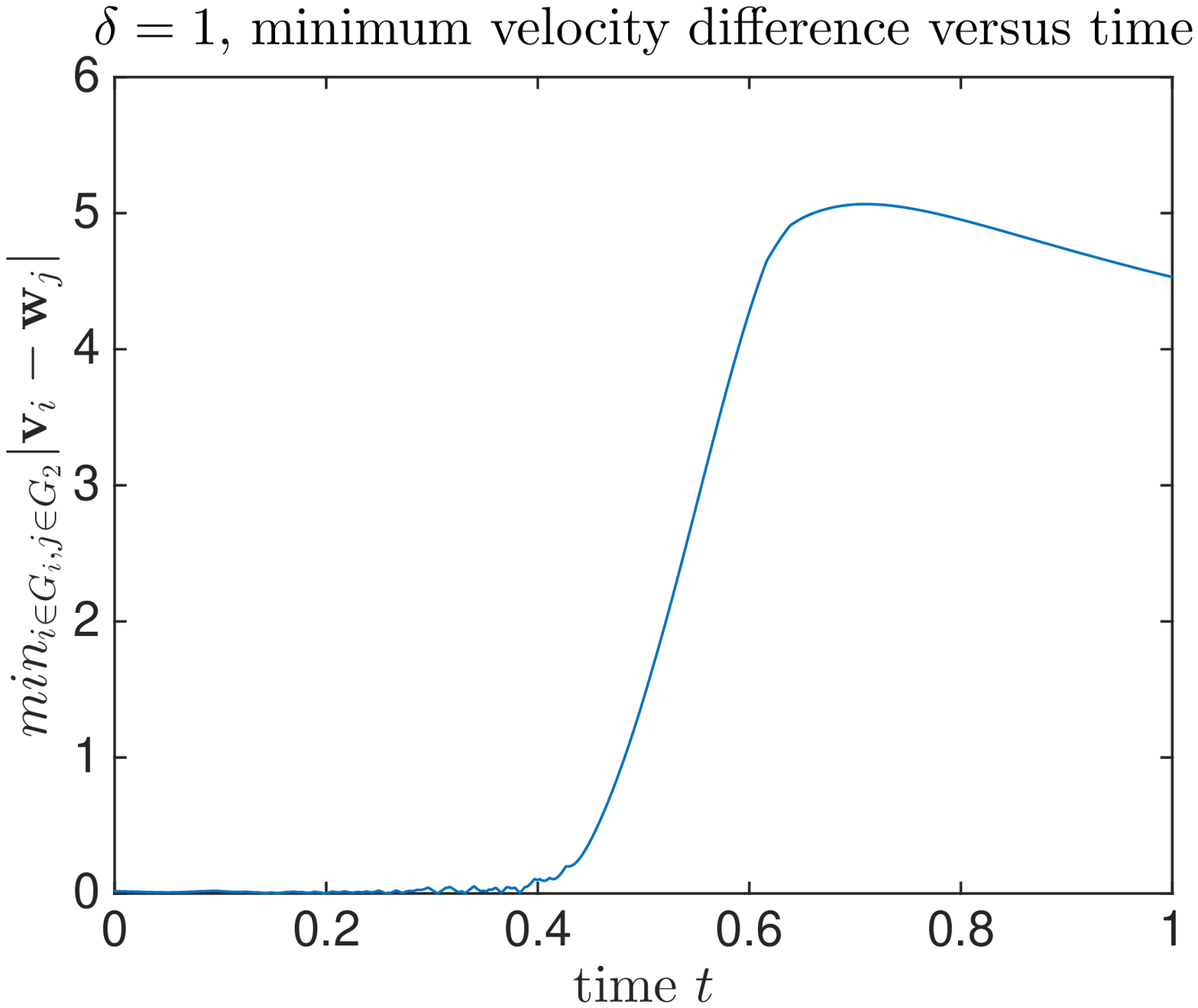}} \\
\subfloat[Stage 2: spatial separation]{\includegraphics[width=0.5\textwidth]{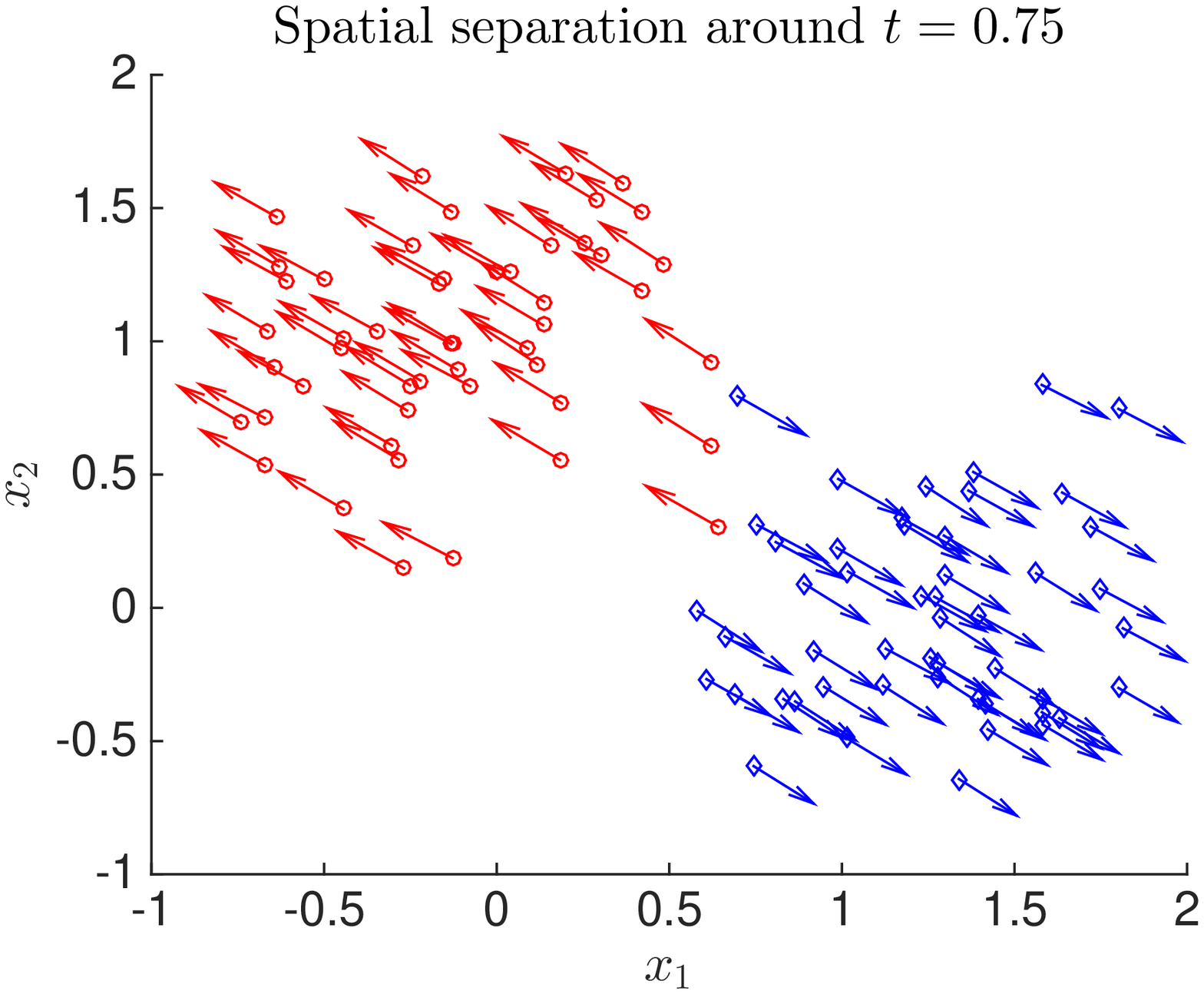}}
\subfloat[Plot of the spatial difference between the two subensembles, where spatial separation happens around $t = 0.7$.]{\includegraphics[width=0.5\textwidth]{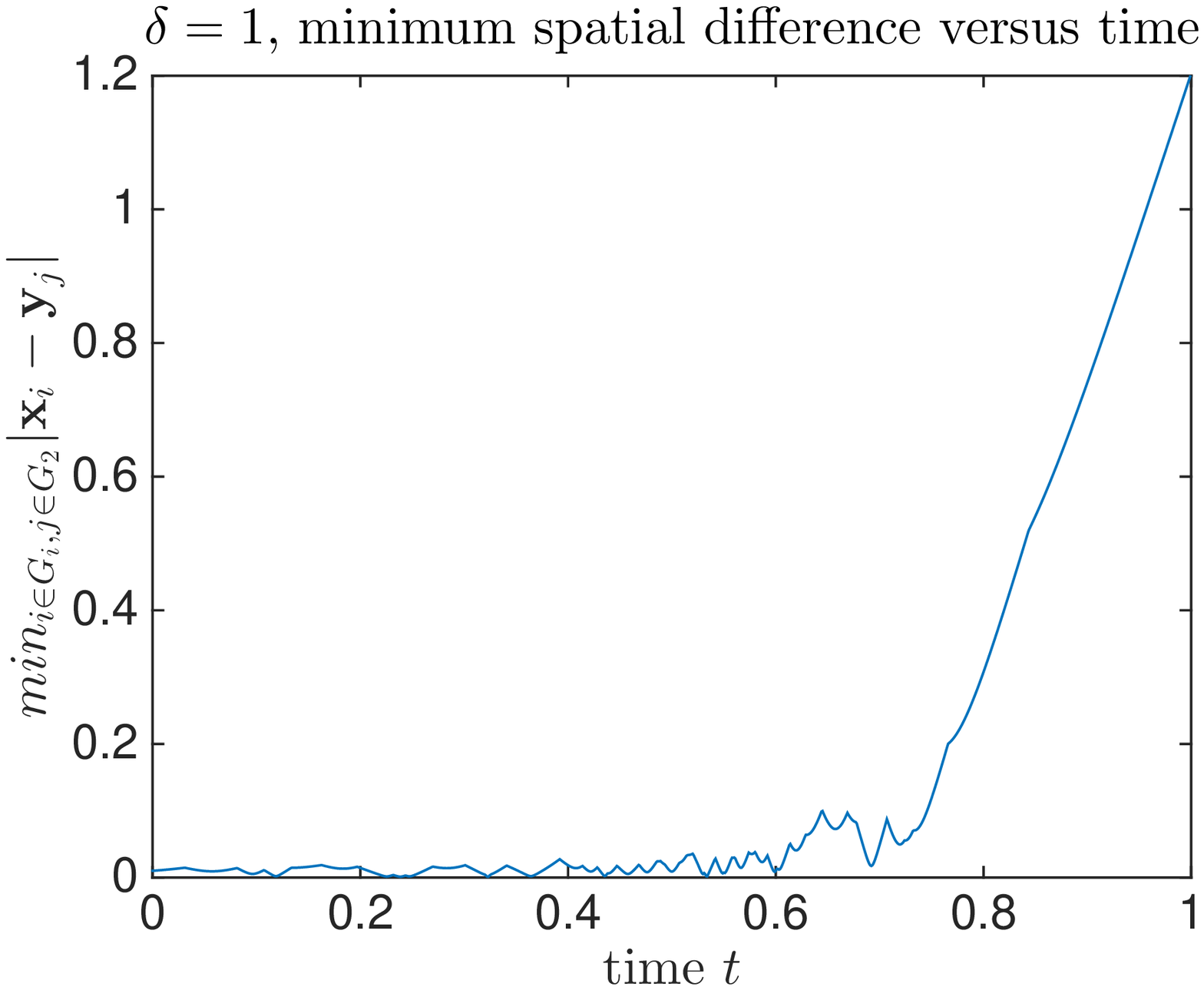}} \\
\subfloat[Stage 3: Emergence of bi-cluster flocking]{\includegraphics[width=0.5\textwidth]{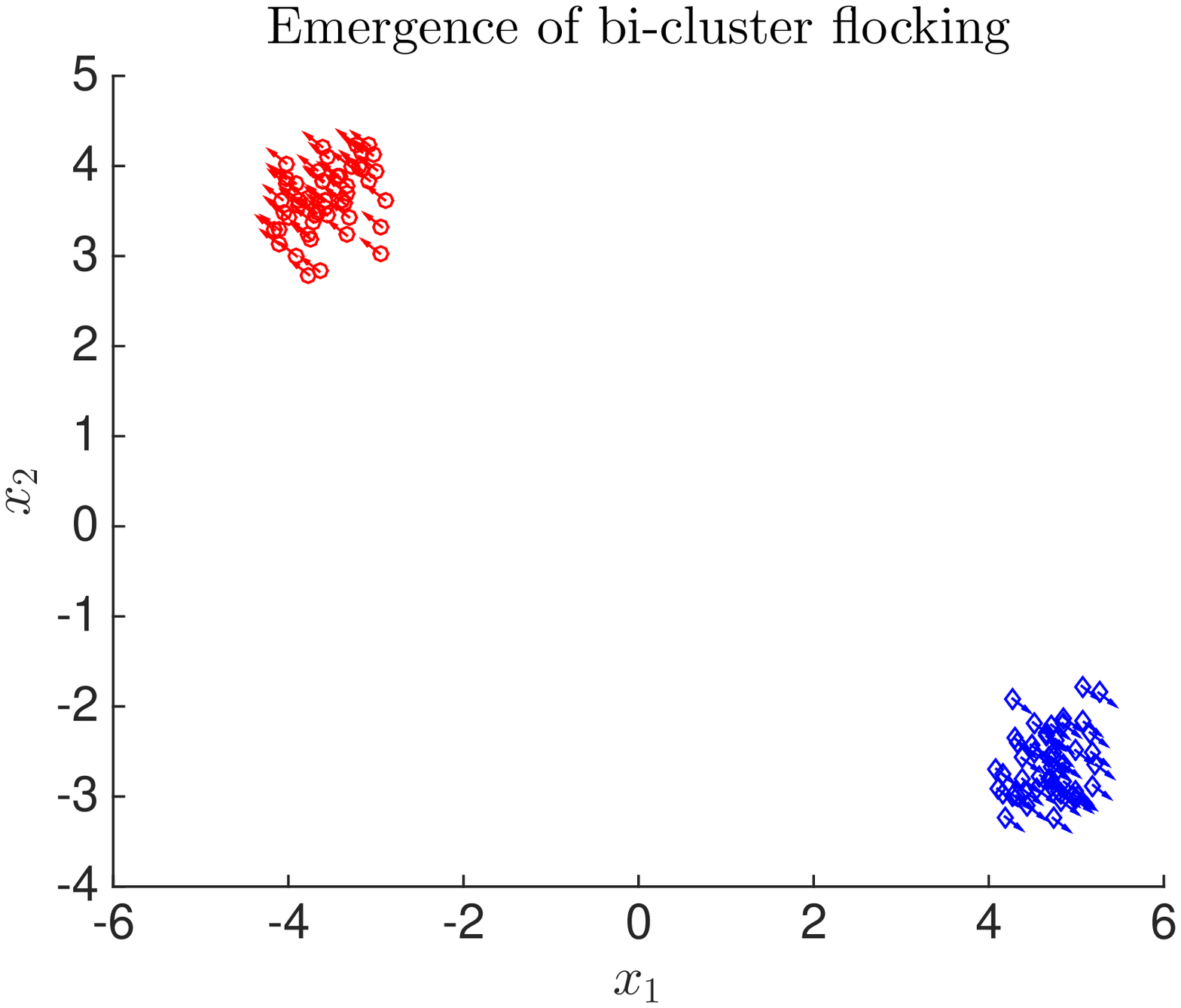}}
\subfloat[Plot of velocity fluctuations]{\includegraphics[width=0.5\textwidth]{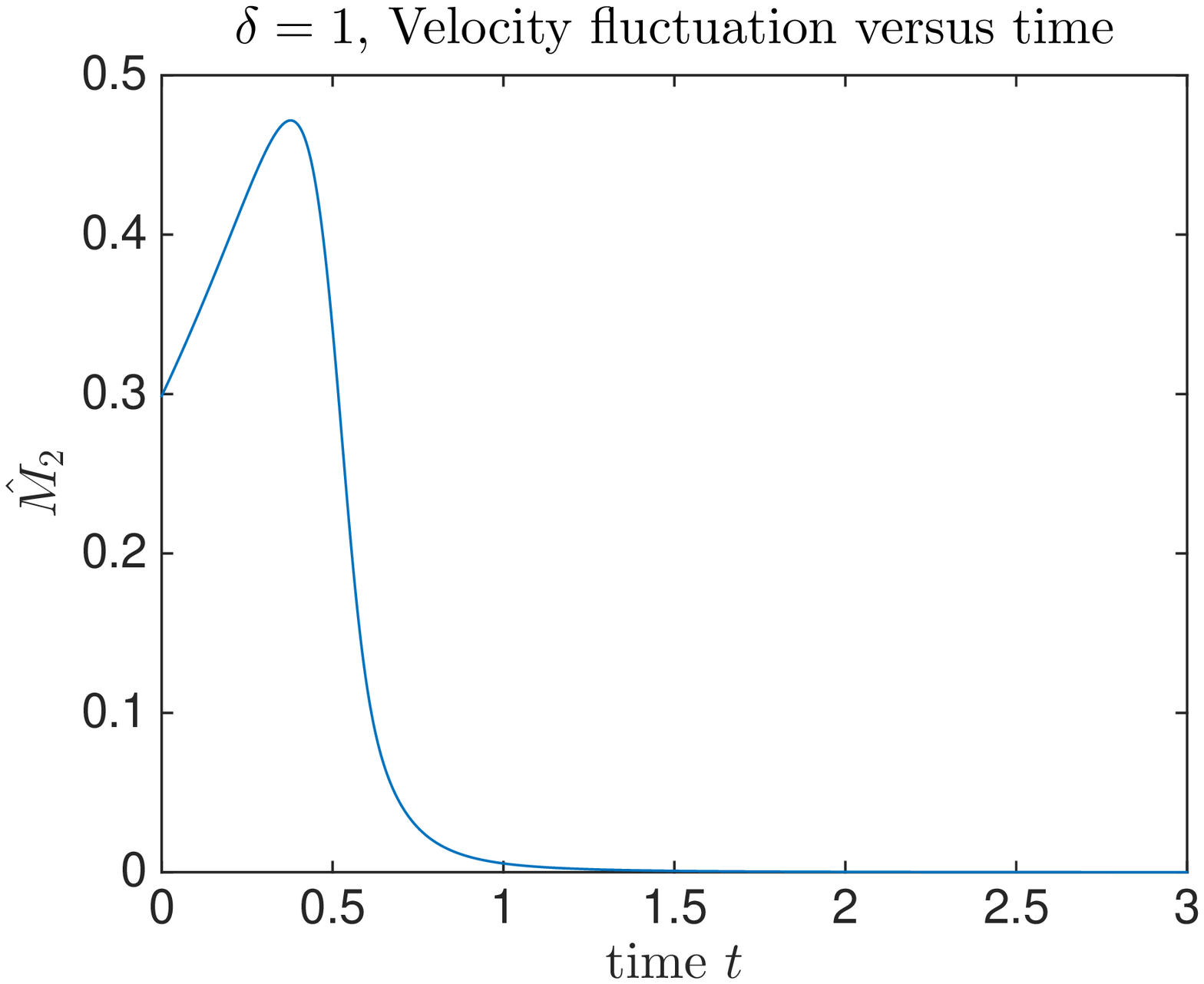}}
\caption{Example \ref{delta1_sd_same}: $\delta = 1$. From mixed intial configuration, the emergence of flocking occurs in three stages as plotted.}
\label{delta1_stage}
\end{figure}
\end{example}

\section{Conclusion} \label{sec:7}
\setcounter{equation}{0}
In this paper, we have studied the emergent dynamics from the interactions between two homogeneous C-S ensembles under the attractive-repulsive couplings and Rayleigh frictions. The interactions among the particles in the same group are assumed to be ``{\it attractive}", whereas the interactions between particles from different groups are assumed to be ``{\it repulsive}". In this situation, we have provided three frameworks for the emergence of bi-cluster flocking from a mixed Cucker-Smale ensemble. As aforementioned in the introduction, prior to the current work, all references dealing with the Cucker-Smale model were focused on the emergent property under only attractive interaction. Thus, we believe that this work is the first step for the understanding of the complex dynamics of the C-S ensemble under attractive and repulsive interactions. 

In the absence of Rayleigh friction, the total kinetic energy can grow exponentially for the C-S ensemble with a general repulsive communication function. This can cause lots of technical difficulties in the flocking analysis. In this work, we assume that the communication mechanism for intra and inter communication weights are different. Under this relaxed setting, we provide two sufficient frameworks leading to the formation of bi-cluster flocking. In the first framework, the communication weight function for inter coupling is a constant. This leads to the complete decoupling of macro dynamics and micro dynamics, moreover the micro-dynamics of each sub-ensemble is also decoupled. Hence we can apply for the previously used Lyapunov functional approach to derive sufficient conditions for the mono-cluster flocking of each sub-ensemble. In the second framework, we take an exponentially decaying function as an inter communication weight function, and provide an a priori setting for the bi-cluster flocking. In the presence of the Rayleigh friction, we  show that the total kinetic energy can be uniformly bounded in time, so some restrictive conditions for the inter and intra communication weight functions in the aforementioned two frameworks can be removed, and when the intra coupling strength is much larger than the inter coupling strength, we show that bi-cluster flocking can emerge from well-prepared spatially mixed configuration. 

Of course, there are several interesting issues that we could not deal with. When inter and inter communication weight functions are comparable, it seems to be very difficult to show the finite-time segregation from spatially mixed configurations.  We will leave this interesting problem for a future work.

\newpage

\appendix

\section{Grownall type lemmas} \label{App-A}
\setcounter{equation}{0}
In this section, we list four Gronwall type lemmas which have been used in the bi-cluster flocking estimates in Section \ref{sec:3} and Section \ref{sec:4}.

\begin{lemma} \label{LA.1}
Let $y: {\mathbb R}_+ \cup \{0 \} \to {\mathbb R}_+ \cup \{0 \}$ be a differentiable function satisfying
\[ y^{\prime} \geq \alpha y - f, \quad t > 0, \qquad y(0) = y_0, \]
where $\alpha$ is a positive constant and $f : {\mathbb R}_+ \cup \{0 \} \rightarrow {\mathbb R}$ is a continuous function decaying to zero as its argument goes to infinity. Then $y$ satisfies
\[  y(t) \geq -\frac{1}{\alpha} \max_{\tau \in [\frac{t}{2},t]} |f(\tau)|  +  \Big( y_0 - \frac{\|f\|_{L^{\infty}}}{\alpha}  \Big) e^{\alpha t} +  \Big( \frac{\|f\|_{L^{\infty}}}{\alpha} + \frac{1}{\alpha}  \max_{\tau \in [\frac{t}{2},t]} |f(\tau)| \Big) e^{\frac{\alpha t}{2}}.\]
\end{lemma}
\begin{proof}
Note that $y$ satisfies
\[ y^{\prime} - \alpha y \geq -f. \]
Multiplying the above differential inequality by $e^{-\alpha t}$ and integrating the resulting relation from $s = 0$ to $s=t$ gives
\begin{eqnarray*}
e^{-\alpha t} y - y_0 &\geq&-\int_0^t f(\tau) e^{-\alpha \tau} d\tau \cr
               &=& -\int_0^{\frac{t}{2}} f(\tau) e^{-\alpha \tau} d\tau - \int_{\frac{t}{2}}^{t} f(\tau) e^{-\alpha \tau} d\tau \cr
               &\geq& -f(0) \int_0^{\frac{t}{2}} e^{-\alpha \tau} d\tau -  \max_{\tau \in [\frac{t}{2},t]} |f(\tau)| \int_{\frac{t}{2}}^{t} e^{-\alpha \tau} d\tau \cr
                &\geq&  \frac{f(0)}{\alpha} \Big( e^{-\frac{\alpha t}{2}} - 1 \Big) -\frac{1}{\alpha}\max_{\tau \in [\frac{t}{2},t]} |f(\tau)|  \Big( e^{-\alpha t} - e^{-\frac{\alpha t}{2}} \Big).
\end{eqnarray*}
Hence,
\[
 y(t) \geq -\frac{1}{\alpha} \max_{\tau \in [\frac{t}{2},t]} |f(\tau)|  +  \Big( y_0 - \frac{f(0)}{\alpha}  \Big) e^{\alpha t} +  \Big( \frac{f(0)}{\alpha} + \frac{1}{\alpha}  \max_{\tau \in [\frac{t}{2},t]} |f(\tau)| \Big) e^{\frac{\alpha t}{2}}.
 \]
\end{proof}
Next, we recall a basic Gronwall type lemma from the bi-cluster flocking paper \cite{C-H-H-J-K} 
\begin{lemma} \label{LA.2}
\emph{\cite{C-H-H-J-K}}
Let $y: {\mathbb R}_+ \cup \{0 \} \to {\mathbb R}_+ \cup \{0 \}$ be a differentiable function satisfying
\[ y^{\prime} \leq - \alpha y + f, \quad t > 0, \qquad y(0) = y_0, \]
where $\alpha$ is a positive constant and $f : {\mathbb R}_+ \cup \{0 \} \rightarrow {\mathbb R}$ is a continuous function decaying to zero as its argument goes to infinity. Then $y$ satisfies
\[ y(t) \leq \frac{1}{\alpha}  \max_{s \in [t/2, t]} |f(s)|+y_0 e^{-\alpha t}+\frac{\|f\|_{L^{\infty}}}{\alpha}e^{-\frac{\alpha t}{2}}, \quad t \geq 0. \]
\end{lemma}
\begin{proof} Since the detailed proof can be found in Lemma A.1 in \cite{C-H-H-J-K}, we omit its proof here. 
\end{proof}

\begin{lemma} \label{LA.3}
Suppose that $\alpha$ and $f$ are nonnegative integrable functions defined on $\bbr_+$. Let $y: {\mathbb R}_+ \cup \{0 \} \to {\mathbb R}_+ \cup \{0 \}$
be a differentiable function satisfying
\[ y^{\prime}(t) \leq \alpha(t) y(t) + f(t), \quad t > 0, \qquad y(0) = y_0, \]
Then $y$ is uniformly bounded: there exists a $y^{\infty}$ such that
\[ \| y \|_{L^{\infty}} \leq  (y_0 + \|f \|_{L^1}) e^{ \|\alpha \|_{L^1}},  \quad t \geq 0. \]
\end{lemma}
\begin{proof} By the method of integrating factor, one has
\[  y(t) \leq y_0 e^{\int_0^t \alpha(s) ds} + \int_0^t e^{\int_s^t \alpha(\tau) d\tau} f(s) ds.  \]
This clearly implies the desired upper bound.
\end{proof}

\begin{lemma} \label{LA.4}
Suppose that $\alpha$ and $f$ be nonnegative integrable functions defined on $\bbr_+$, and ;et $y: {\mathbb R}_+ \cup \{0 \} \to {\mathbb R}_+ \cup \{0 \}$
be a differentiable function satisfying
\[ y^{\prime}(t) \geq \alpha(t) y(t) - f(t), \quad t > 0, \qquad y(0) = y_0, \]
Then $y$ satisfies
\[  y(t)  \geq y_0 -  ||f||_{L^1} e^{||\alpha||_{L^1}} \quad t \geq 0.  \]
\end{lemma}
\begin{proof} By the method of integrating factor, one has
\[  y(t) \geq y_0 e^{\int_0^t \alpha(s) ds} + \int_0^t e^{\int_s^t \alpha(\tau) d\tau} f(s) ds \geq y_0 -  ||f||_{L^1} e^{||\alpha||_{L^1}}.  \]
\end{proof}

\end{document}